\documentclass{article}

\usepackage{arxiv}

\usepackage[utf8]{inputenc} 
\usepackage[T1]{fontenc}    
\usepackage{hyperref}       
\usepackage{url}            
\usepackage{booktabs}       
\usepackage{amsmath,amssymb,amsthm,mathrsfs}       
\usepackage{nicefrac}       
\usepackage{microtype}      
\usepackage{graphicx}
\usepackage{natbib}
\usepackage{doi}

\usepackage{lettrine}
\usepackage[linesnumbered,ruled,vlined]{algorithm2e}
\usepackage{subcaption}
\usepackage{multirow}
\SetKwInput{KwDef}{Define}               
\SetKwInput{KwInput}{Input} 
\newtheorem{theorem}{Theorem}
\newtheorem{corollary}{Corollary}[theorem] 

\title{Systems of ODEs Parameters Estimation by Using Stochastic Newton-Raphson and Gradient Descent Methods}

\author{ \href{https://orcid.org/0000-0003-4526-8265}{\includegraphics[scale=0.06]{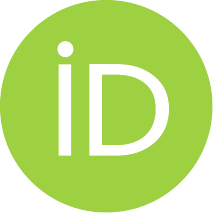}\hspace{1mm}S.~Syafiie}\thanks{corresponding author} \\
	Department of Chemical \\and Materials Engineering,\\
	Faculty of Engineering - Rabigh,\\
	King Abdulaziz University,\\ Jeddah, Kingdom of Saudi Arabia \\
	\texttt{s.syafiie@gmail.com} \\
	\And
	\href{https://orcid.org/0009-0008-1996-5124}{\includegraphics[scale=0.06]{orcid.pdf}\hspace{1mm}Aries Subiantoro} \\
	Department of Electrical Engineering,\\
	Faculty of Engineering,\\
	Universitas Indonesia,\\ Depok - Jawa Barat, Indonesia\\
	\texttt{aries.subiantoro@ui.ac.id} \\
	 \AND
	 \href{https://orcid.org/0000-0001-7627-8875}{\includegraphics[scale=0.06]{orcid.pdf}\hspace{1mm} Vivi Andasari} \\
	 Department of Electrical Engineering,\\
	Faculty of Engineering,\\
	Universitas Indonesia,\\ Depok - Jawa Barat, Indonesia\\
	 \texttt{andasarivivi@gmail.com} \\
	 \And
	 \href{https://orcid.org/0000-0002-4046-3136}{\includegraphics[scale=0.06]{orcid.pdf}\hspace{1mm} Fernando Tadeo }\\
	 Departamento de Ingenieria \\de Sistemas y Automatica,\\ Universidad de Valladolid,\\ Valladolid, Spain \\
	 \texttt{fernando@autom.uva.es} \\
}
\renewcommand{\shorttitle}{System of ODEs parameters estimation}

\hypersetup{
pdftitle={A template for the arxiv style},
pdfsubject={q-bio.NC, q-bio.QM},
pdfauthor={David S.~Hippocampus, Elias D.~Striatum},
pdfkeywords={First keyword, Second keyword, More},
}

\begin{document}
\maketitle

\begin{abstract}
	Ordinary differential equations (ODEs) are widely used to describe the time evolution of natural phenomena across various scientific fields. Estimating the parameters of these systems from data is a challenging task, particularly when dealing with nonlinear and high-dimensional models. In this paper, we propose novel methodologies for parameter estimation in systems of ODEs by using the Newton-Raphson (NR) method and Gradient Descent (GD) method. By leveraging the discrete derivative and Taylor expansion, the problem is formulated in a way that enables the application of both methods, allowing for flexible, efficient solutions. Additionally, we extend these approaches to stochastic versions — Stochastic Newton-Raphson (SNR) and Stochastic Gradient Descent (SGD) — to handle large-scale systems with reduced computational cost. The proposed methods are evaluated by using numerical examples, including both linear and nonlinear parameter models, and compare the results to the well-known Nonlinear Least Squares (NLS) method. While NR converges rapidly to the optimal solution, GD demonstrates robustness in handling chaotic systems, though it may occasionally lead to suboptimal results. Overall, the proposed methods provide improved accuracy in parameter estimation for ODE systems, outperforming NLS in terms of error metrics such as bias, mean absolute error (MAE),  mean absolute percentage error (MAPE), root mean square error (RMSE), and  coefficient of determination R². These methods offer a valuable tool for fitting ODE models, particularly in scenarios involving big data and complex dynamics.
\end{abstract}


\section{Introduction}
\lettrine[nindent=0em,lines=3]{A}{}
well presented physical phenomenon mathematically is described by an ODE or a systems of ODEs. Thus, ODEs are a very important powerful tool to formulate, study and analyze physical behavior of a system. The ODEs are extensively employed in science and technology for example in physics, chemistry, biology, mechanics, astronomy, economics, etc. The famous ODE is due the objective laws which governs certain phenomena can be described as an ODE or a system of ODEs. Furthermore, the ODEs solution is given a better future prediction.\\
ODE is an equation having only one instantaneous derivative variable which respects to a single independent variable. Whereas partial differential equation (PDE) is an equation having a single dependent variable which respect to several independent variables. Most physical systems may experience that the previous dependent variable affects the current measurement of dependent variables. This experiencing phenomenon is well described in  a delay differential equation (DDE). \\
Also, ODE provides a direct representation of a physical system evolution while allowing the size sample step in irregular space. In an interest for experimental data collection, the sampling interval may vary for each collecting time. Therefore, this article aims to fit a given data set to a system of ODEs model.\\
A curve fitting is a process of identifying parameters of the mathematical models that gives fittest curve to the given data sets. Fitting to a dynamical system is a challenging task. Fitted curves can be implemented as a tool for data visualization to infer the query where no data are available, with possibility to seek value beyond the range of the observed data. \\
An ODE represents the physical system behavior. It is usually described mathematically as
\begin{equation}\label{eq01}
\frac{dx}{dt}=f(t,x,a), \quad x(0)=x_0,
\end{equation}
where $x$ is a physical system variable measured at time $t$, $f(\cdot)$ is the function of the system evolution, $a$ is time-independent parameter and $x_0$ is an initial condition. \\
Usually the parameter is estimated based on different of derivative $\frac{dx(t)}{dt}$ to function $f(t,x,a)$ in sum squared error (SSE) of  \smash{$\|\frac{d\hat{x}(t)}{dt}-f(t,\hat{x},\bar{a})\|$}, where $\hat{x}$ is approximated variable $x(t)$ and $\bar{a}$ is estimated time-independent parameter $a$. 
In order to search for an unknown parameter, the ODE in equation (\ref{eq01}) has to be solved or approximated. 
Most suggested approaches are of proposed explicit solution to the ODE. Based on linearized least square minimization problem
of the defined state and measurement vector \citep{Hwang1972} and expanded measurement vector by span state-space to have eigenvectors of weighted matrix \citep{Gavalas} have been suggested for linear ODEs.   
Having explicit solutions in term of known elementary or transcendental functions has been suggested by \cite{Howland61}.
Also, it has been  suggested by \cite{Strebel} to determine tangent slope and coordinates for solution of ODE. Adjoint sensitivity analysis is used to estimate ODE solution by \cite{Aydogmus}. As well, it is approximated by introducing integral estimation based on Bayesian collocation approach to avoid expensive numerical solution of the ODE \citep{Xu}. Having ODE as an equality constraint for quadratic programming with selecting minimization problem of measurement vector as a cost function has further suggested by \cite{Li2005}.\\
Using a spline straightforward least squares approach  proposed by \cite{Varah1982}, introduced finite set of orthogonal constraints to variational characterization of the solution \citep{Brunel2014},
penalized splines used for estimating the time-varying coefficients in an ODE based on variation of the nonlinear least squares
\citep{Cao2012} have been suggested. A method for parameter estimation using the Kalman's filter has been putted into the bargain. With appropriate initial conditions, the filter solution can approximate the minimum-norm weighted least-squares solution of the estimation \citep{Aidala-1977}. Signals described by linear differential equations whose unknown parameters are estimated by using an orthogonal projection  \citep{Mboup2008}.\\
Also parameter estimation for ODE  has been used a
multiple shooting method by \cite{Peifer2007} and
modification of data smoothing method along with a generalization of profiled estimation by using manifold solution typical dimension of differential equation \citep{Ramsay2007}. By presented ODE (in equation (\ref{eq01})) in multipoint boundary value problem in discretized for nonlinear constraint optimization problem \citep{Baake1992}, it has been implied.\\
Parameters estimation for stochastic differential equations by using maximum likelihood and quasi maximum likelihood estimation
\citep{Timmer2000}, Markov chain Monte Carlo methods \citep{Mbalawata2013, Chen2021} for linear and non-linear It\^{o} type stochastic differential equations have been hinted. Partially matching empirical expected signature (using Monte Carlo approximation) of the observed path has been suggested to estimate unknown parameters \citep{Papavasiliou2011}. Furthermore, comparison study of several estimated method for nonlinear stochastic differential equations with discrete time measurement has been conducted by \cite{Singer2002}. Surveyed and discussed of parameters estimation for stochastic differential equation have been proposed by \cite{Nielsen2000}. Also, interesting in estimating unknown parameters for partial differential equation has been suggested by 
a parameter cascading method and a Bayesian approach \citep{Xun2013}. Similar interest is also for estimating unknown parameters for time varying delay differential equation such as using least squares support vector machines \citep{Mehrkanoon2014}. Parameterisation obtained from the power balance equation of Euler-Lagrange systems can be used for their identification and adaptive control even in a scenario with insufficient excitation by having a modification linear regression equation \citep{Romero2021}.\\
Artificial neural network (ANN) has also been applied for parameters estimation of ODEs. Such a technique is based on two subproblems: first subproblem generates to have mapping model of the given data and other subproblem uses the model to obtain an estimate of the parameters \citep{Dua2011} and simultaneously minimization \citep{Dua2012}. Applying machine learning models as a tool to fit nonlinear mechanistic ODEs has been putted by \cite{Bradley2021} by using two-stage indirect approach. Similarly, ANN has also proposed by \cite{Jamili} to estimate parameters of partial differential equation, applying particle swarm optimization has also been shown by \cite{Akman} as a candidate to estimate unknown parameter for ODE .  \\
GD method has been used for numerical solution of ODE by \cite{Neuberger1985}. Two block coordinate descent algorithm has been applied for optimization problem with ODE as dynamical constrains by taking measured vector and state vector as cost function by \cite{Matei}. Modified momentum GD has been applied for parameters estimation of ARX model by \cite{Tu2020}. \\
Unlike previous methods discussed above that the minimization problem between state vector and measurement vector was used, in this article the problem is formulated by considering the discrete form of the derivative part. Then, the nonlinear function is approximated by using a well-known Taylor's expansion. Thus, the formulated problem can be solved directly by using such as NR method, GD method, etc. This direct formulated problem in parameters estimation is a main contribution of this article. To check the applicability, the proposed algorithms are applied to estimate parameters of systems of ODEs which have linear and nonlinear in parameters. The proposed algorithm can be easily expanded to stochastic NR and GD methods, which can be applied for big data. 
Hence, parameters obtained by using both SNR and SGD methods are analyzed and compared to well-known nonlinear least square method. \\
The remaining of this article is structured and organized as the following: The methodology development is given in section \ref{methodology}, explaining, discussion and comparison to the well known nonlinear least square of 3 numerical examples are posted in section \ref{Num_example}. Then, the proposed algorithms are remarked in conclusion section \ref{conclusion}.

\section{Methodology Development}\label{methodology}

Suppose a given data set $(t_d,x^d_k)$ for $d=1,\cdots,N$ and $k=1,\cdots, n$ to be fitted to a system of ODEs
\begin{equation}\label{eq03} 
\frac{dx_j(t)}{dt}=f_j(t,x_k,a_l),\quad j,k=1,\cdots,n,\quad \quad l=1,\cdots,m,
\end{equation}
where $x_k$ is the measured states at time $t$, $a_l$ is unknown parameters. For simplicity, it is presented as 
\begin{equation}\label{eq04}
\frac{dx(t)}{dt}=F(t,x,a),
\end{equation}
where $F=f_j$, $x=x_k$ and $a=a_l$.\\
A function $x(t)$ is approximated by Taylor expansion as
\begin{equation}
x(t)|_{t\approx \tau}=\sum_{n=0}^\infty \frac{x^{(n)}(\tau)}{n!}\left(t-\tau\right)^n.
\end{equation}
Let has $t-\tau=t_{i+1}-t_{i}$, then
for $n=1$, it has
\begin{equation}\label{eqT1}
x(t_{i+1})=x(t_i)+x'(t_i)(t_{i+1}-t_{i})+\frac{(t_{i+1}-t_i)^2}{2}x(c),
\end{equation}
thus, it gives
\begin{equation}
x'(t_i)=\frac{x(t_{i+1})-x(t_i)}{t_{i+1}-t_i}.
\end{equation}
Thus $x'(t_i)$ has to fulfill Lipschitz continuity
\begin{equation}
\|x(t_{i+1})-x(t_i)\|\le L\|t_{i+1}-t_i\|,
\end{equation}
for $L$ is Lipschitz's constant. More accurate derivative of $x(t_i)$ can be approximated by having higher order Taylor series, such as for $n=2$, then
\begin{equation}\label{eqT2}
x(t_{i+1})=x(t_i)+x'(t_i)(t_{i+1}-t_{i})+\frac{x''(t_i)}{2!}(t_{i+1}-t_i)^2.
\end{equation}
To estimate the $x''(t_i)$, it is by approximating $x(t_{i+2})$ from $x(t_i)$ as 
\begin{equation}\label{eqT3}
x(t_{i+2})=x(t_i)+x'(t_i)(t_{i+2}-t_{i})+\frac{x''(t_i)}{2!}(t_{i+2}-t_i)^2.
\end{equation}
Let have a different of equations (\ref{eqT3}) and (\ref{eqT2}) as
\begin{equation}
x(t_{i+2})-x(t_{i+1})=x'(t_i)(t_{i+2}-t_{i+1})+\frac{x''(t_i)}{2!}\left((t_{i+2}-t_i)^2-(t_{i+1}-t_i)^2 \right),
\end{equation}
which gives
\begin{equation}\label{eqT4}
x''(t_i)=2\frac{x(t_{i+2})-x(t_{i+1})-x'(t_i)(t_{i+2}-t_{i+1})}{(t_{i+2}-t_i)^2-(t_{i+1}-t_i)^2}.
\end{equation}
Substitute equation (\ref{eqT4}) into equation (\ref{eqT2}) will gives
\begin{equation}\resizebox{0.9\hsize}{!}{
$x'(t_i)=\frac{(x(t_{i+1})-x(t_i))((t_{i+2}-t_i)^2- (t_{i+1}-t_i)^2)-(x(t_{i+2})-x(t_{i+1}))(t_{i+1}-t_i)^2}{(t_{i+1}-t_i)((t_{i+2}-t_i)^2- (t_{i+1}-t_i)^2)-(t_{i+2}-t_{i+1})(t_{i+1}-t_i)^2}$},
\end{equation}
for case $t_{i+2}-t_{i+1}\ne t_{i+1}-t_i$. 
It is no necessary that $t_{i+1}-t_i$ or $t_{i+2}-t_{i+1}$ to be constant. 

\subsection{Newton-Raphson Method}

The function $f(\cdot)$ of equation (\ref{eq01}) can be approximated by using a well-known Taylor expansion at $\bar{a}$ as
\begin{equation}
\bar{f}(t,x,\bar{a})=f(t,x,\bar{a})+\nabla_a f(t,x,\bar{a})\Delta a +O(\|\Delta a\|^2),
\end{equation}
where $\Delta a=a-\bar{a}$. Thus, equation (\ref{eq01}) can be rewritten as
\begin{equation}\label{NR01}
e(t,x,\bar{a})=-\nabla_a f(t,x,\bar{a})\Delta a,
\end{equation}
where $e(t,x,\bar{a})=f(t,x,\bar{a})-\frac{x(t+\Delta t)-x(t)}{\Delta t}$. Clearly, it is of finding parameter $\bar{a}$ such that the error $e(t,x,\bar{a})=0$. Thus the equation (\ref{NR01}) can be presented in general as
\begin{equation}\label{NR02}
E(t,x,\bar{a})=-\nabla_a F(t,x,\bar{a})\Delta a,
\end{equation}
where $\nabla_a F(t,x,\bar{a})$ is a Jacobian matrix. Hence the unknown parameters can be obtained recursively as
\begin{equation}\label{NR03}
a_{i+1}=a_{i}-\gamma\nabla_aF(t,x,a_i)E(t,x,a_i),
\end{equation}
where
\[\gamma=(\nabla_aF(t,x,a_i)^\top\nabla_aF(t,x,a_i))^{-1},\]
by considering that $\Delta a=a_{i+1}-a_i$. This technique is known as Newton-Raphson (NR) method.\\
\begin{algorithm}[!h]
\DontPrintSemicolon
	\KwData{Data set, $(t_d,x^d_k)$}
	\KwDef{$\frac{dx}{dt}=F(t,x,a)$}
	\KwInput{initial guess, $a_0$}
	\KwInput{tolerable error, $\epsilon$}
	\While{$|a_{i}-a_{i-1}|>\epsilon$}
	{calculate $E$ and $\nabla_aF$ \\
	 update $a_i$ following equation (\ref{NR03})
	}
\caption{Implementation}\label{algorithm}
\end{algorithm}
\subsection*{Stochastic NR Method}
To reduce computational cost for large-scale system of ODEs, a stochastic NR (SNR) method can be applied by solving part of the original system of ODEs. It is by randomly selected $r$ number of equations from $F(\cdot)$ for each iteration. In introduction to stochasticity, a random variable is defined $\pi$ on a probability space $(\Phi,\mathcal{F,P}):\pi:\Phi\to \mathcal{R}$, where $\mathcal{R}$ is a set of all combination of $r$ numbers out of $n$ ODEs. Since $|\mathcal{R}|=\binom{n}{r}$. Then let $\mathcal{R}=\{\rho_1, \rho_2, \cdots, \rho_{|\mathcal{R}|}\}$ with the assumption is that $\pi$ follows uniform distribution, $\mathcal{P}(\pi=\rho_p)=\frac{1}{|\mathcal{R}|}$ for $1 \le p \le |\mathcal{R}|$. Thus, it can be seen that $F(t,x,a,\rho_p):=[f_{j1}(t,x,a_l), f_{j2}(t,x,a), \cdots, f_{jr}(t,x,a)]^\top$ where $\rho_p=\{j1, \cdots,jr\}\subset \{1, \cdots, n\}$. Then equation (\ref{NR03})
can presented in SNR as
\begin{equation}\label{SNR01}
a_{i+1}=a_{i}-\gamma\nabla_aF(t,x,a_i,\rho_p)E(t,x,a_i,\rho_p),
\end{equation}
where
\[\gamma=(\nabla_aF(t,x,a_i,\rho_p)^\top\nabla_aF(t,x,a_i,\rho_p))^{-1}.\]
The algorithm for the stochastic Newton-Raphson Method is presented as
\begin{algorithm}[!h]
\DontPrintSemicolon
	\KwData{Data set, $(t_d,x^d_k)$}
	\KwDef{$\frac{dx}{dt}=F(t,x,a)$}
	\KwInput{initial guess, $a_0$}
	\KwInput{tolerable error, $\epsilon$}
	\While{$|a_{i}-a_{i-1}|>\epsilon$}
	{select $\rho_p$\\
	 calculate $E$ and $\nabla_aF$ \\
	 update $a_i$ following equation (\ref{SNR01})
	}
\caption{Implementation}\label{algorithm_sNR}
\end{algorithm}

\subsection*{Convergence analysis}
\begin{theorem}
Let $g_j(t,x,a):=f_j(t,x,a)-\frac{dx}{dt}$ be twice continuously differentiable, by considering that $\nabla_ag_j(t,x,a)\ne 0$, thus by sequence
\begin{equation}
a_{i+1}=a_i-(\nabla_a g_j(t,x,a_i))^{-1}g_j(t,x,a_i),
\end{equation}
converges to $a^*$ as $i\to\infty$. Then for $i$ sufficient large, 
\begin{equation}
|a_{i+1}-a^*|\le \Omega |a_i-a^*|,\quad \mbox{if}\; \Omega>
\left|\nabla_a g(t,x,a^*)\right|^{-1}\left|\sum_{j=1}^n \mu^\top_iH_j(t,x,a^*)\mu_i I_j\right|,
\end{equation}
where $H_j(t,x,a^*)$ is Hessian matrix of $g(t,x,a^*)$ and $I_j$ is the standard basis of $\mathbb{R}^n$.
Thus $a_i$ converges to $a^*$ quadratically. 
\end{theorem}
\begin{proof}
Let difference for each iteration to optimal solution be $\mu_i=a_i-a^*$, then $a_i-\mu_i=a^*$. Thus the roots of equation (\ref{eq03}) can be presented as
\begin{equation}
g_j(t,x,a):=f_j(t,x,a)-x'(t)=0,
\end{equation}
for approximate discrete $\frac{dx_j}{dt}$. Applying Taylor's theorem to $g_j(\cdot)$ by setting $a=a_i$ and
\begin{equation}
g_j(t,x,a_i-\mu_i)=g_j(t,x,a_i)-(\mu_i)\nabla_a g(t,x,a_i)+\sum_{j=1}^n \mu^\top_iH_j(t,x,a_y)\mu_i I_j,
\end{equation}
for $a_y$ is between $a_i$ and $a^*$. By considering that at an optimal solution $a^*$ is a perfect fitting, and $a_i-\mu_i=a^*$ thus
\begin{equation}
0=g_j(t,x,a_i)-(a_i-a^*)\nabla_a g_j(t,x,a_i)+\sum_{j=1}^n \mu^\top_iH_j(t,x,a_y)\mu_i I_j.
\end{equation}
Since $g_j(\cdot)$ and $\nabla_a g_j(\cdot)$ are continuous and $\nabla_a g_j(t,x,a^*)\ne 0$, for $a_i$ is close enough to $a^*$. Thus
\begin{equation}\resizebox{0.9\hsize}{!}{$
0=(\nabla_a g_j(t,x,a_i))^{-1}g_j(t,x,a_i)-(a_i-a^*)+\nabla_a (\nabla_a g_j(t,x,a_i))^{-1}\sum_{j=1}^n \mu^\top_iH_j(t,x,a_y)\mu_i I_j$},
\end{equation}
which gives
\begin{equation}
a_{i+1}-a^*=|\nabla_a g_j(t,x,a_i)|^{-1}\sum_{j=1}^n \mu^\top_iH_j(t,x,a_y)\mu_i I_j,
\end{equation}
thus
\begin{equation}
|a_{i+1}-a^*|\le |\nabla_a g_j(t,x,a_i)|^{-1}|\sum_{j=1}^n \mu^\top_iH_j(t,x,a_y)\mu_i I_j||a_i-a^*|^2.
\end{equation}
In long run, the $\nabla_a g(t,x,a_i)$ will converge to $\nabla_a g(t,x,a^*)$, by having that $a_y$ in between $a_i$ and $a^*$, thus $a_y$ may converge to $a^*$, therefore $\sum_{j=1}^n \mu^\top_iH_j(t,x,a_y)\mu_i I_j$ converges to $\sum_{j=1}^n \mu^\top_iH_j(t,x,a^*)\mu_i I_j$
\begin{equation}
|a_{i+1}-a^*|\le \Omega |a_i-a^*|,\quad \mbox{if}\; \Omega>|\nabla_a g_j(t,x,a_i)|^{-1}\left|\sum_{j=1}^n \mu^\top_iH_j(t,x,a^*_l)\mu_i I_j\right|.
\end{equation}
\end{proof}

\begin{corollary}
Equation (\ref{NR02}) can be presented  in the standard NR method as
\begin{equation}
a_{i+1}=a_i-(\nabla_a F(t,x,a_i))^{-1}E(t,x,a_i).
\end{equation}
As $E(t,x,a_i)\in\mathbb{R}^{n(N-1)\times l}$ and $\nabla_a F(t,x,a_i)\in\mathbb{R}^{n(N-1)\times l}$, then, it can be presented as in equation (\ref{NR03}). 
\end{corollary}

\subsection{Gradient Descent Method}
Systems in equation (\ref{eq04}) can also be presented as
\begin{equation}\label{GD01}
G(t,x,a)=F(t,x,a)-\frac{dx}{dt}=0.
\end{equation}
Then, it can be solved by solving minimization problem of the following
\begin{equation}
\min_a\frac{1}{2}\|G(t,x,a)\|^2_2,
\end{equation}
which can also solve by using a gradient descent (GD) method as
\begin{equation}\label{GD02}
a_{i+1}=a_i-\eta\nabla_a G(t,x,a_i)G(t,x,a_i),
\end{equation}
where $\nabla_a G(t,x,a_i)$ is the Jacobian matrix at $a_i$  which is similar to $\nabla_a F$ in equation (\ref{NR02})
and $\eta$ is a step size (also called learning rate). As step size plays an important rule in developing efficient DG method, many researchers have suggested methodologies on computing the step size such as fixed step size, backtracking line search, exact line search, two - point (also called Barzilai and Borwein method \citep{Barzilai}),  alternate step \citep{Dai2003}, extended Cauchy-Barzilai-Borwein search \citep{Raydan},  explicit choose \citep{Hao}, etc. \\
In this article, the step size is updated explicitly as
\begin{equation}\label{l_rate}
\eta=\frac{\delta^\top G(t,x,a_i)}{\delta^\top \delta},\quad \delta=\nabla_a G(t,x,a_i)\nabla G(t,x,a_i)^\top G(t,x,a_i),
\end{equation}
which is chosen based on considering $a_{i+1}=a_i-\eta v_i$, where $v_i=\nabla_a G(\cdot)^\top G(\cdot)$. Then, by applying Taylor expansion to the problem is to have $G(t,x,a_{i+1})\approx G(t,x,a_{i})-\eta\nabla_a G(t,x,a_i)v_i\approx 0$, which implies equation (\ref{l_rate}) for any given $\delta$. However, it is selected $\delta$ as in the equation (\ref{l_rate}) due to guarantee the convergence \cite{Hao} such that 
\begin{equation}
G(t,x,a_{i+1})^\top G(t,x,a_{i+1})\le G(t,x,a_i)^\top G(t,x,a_i).
\end{equation}
It can be implemented by using the following algorithm.
\begin{algorithm}[!h]
\DontPrintSemicolon
	\KwData{Data set, $(t_d,x^d_k),\; d=1,\cdots,N$}
	\KwInput{initial guess, $a_0$}
	\KwDef{$G$, $\nabla_a G$, $\delta$}
	\KwInput{tolerable error, $\epsilon$}
	\While{$|a_i-a_{i-1}|>\epsilon$}
	{calculate $\eta$ and $\nabla_a G^\top G$ \\
	 update $a^i$ following equation (\ref{GD02})
	}
\caption{Implementation}\label{algorithm2}
\end{algorithm}
\subsection*{Stochastic Gradient Descent}
As mentioned above for SNR, it is also possible to use a stochastic gradient descent (SGD) method for reducing the computational cost by solving a part of the original ODEs. With similar definition of SNR, SGD  of equation (\ref{GD02}) can be updated as
\begin{equation}\label{sGD}
a_{i+1}=a_i-\eta\nabla_a G(t,x,a_i,\rho_p)G(t,x,a_i,\rho_p).
\end{equation}
It can applied by using the following algorithm.
\begin{algorithm}[!h]
\DontPrintSemicolon
	\KwData{Data set, $(t_d,x^d_k),\; d=1,\cdots,N$}
	\KwInput{initialize, $a_0$, $\rho_p$}
	\KwDef{$G$, $\nabla_a G$, $\delta$}
	\KwInput{tolerable error, $\epsilon$}
	\While{$|a_i-a_{i-1}|>\epsilon$}
	{select $\rho_p$\\
	 calculate $\eta$ and $\nabla_a G^\top G$ \\
	 update $a^i$ following equation (\ref{sGD})
	}
\caption{Implementation}\label{algorithm22}
\end{algorithm}

\subsection*{Convergence analysis}

\begin{theorem}
A function $G(t,x,a)$ fulfilled Lipschitz's condition having Jacobian $\nabla_a G(t,x,a)$ which rank-$b$ matrix close to $a^*$, updating $a_i$ converges to $a^*$ linearly from initial guess $a_0$ if $a_0$ is the neighborhood of $a^*$.
  
\end{theorem}

\begin{proof}
$G(t,x,a)$ is said to be a smooth function in neighborhood of $a^*$ if its gradient are Lipschitz continuous. Thus, $G(t,x,a)$ can be approximated by a Taylor expansion as
\begin{equation}\resizebox{0.9\hsize}{!}{$
G(t,x,a_i-\mu_i)=G(t,x,a_i)+\nabla_a G(t,x,a_i)\mu_i+\sum_{j=1}^n (\mu_i)^\top H_j(t,x,a_y)\mu_iI_j+O(\|\mu_i\|^3)$},
\end{equation}
\begin{equation}
\nabla_a G(t,x,a_i-\mu_i)=\nabla_a G(t,x,a_i)+\sum_{j=1}^n I_j(H_j(t,x,a_y))^\top + O(\|\mu_i\|^2),
\end{equation}
where $H(t,x,a_y)$ is Hessian of $G(t,x,a_y)$. Recall GD searching with the expansion for 
\begin{equation}
\mu_{i+1}=\mu_i -\eta\nabla_a G(t,x,a_i)^\top G(t,x,a_i)\mu_i+\sum_{j=1}^n (\mu_i)^\top H(t,x,a_y)\mu_iI_j+O(\|\mu_i\|^3),
\end{equation}
collecting to be
\begin{equation}
\mu_{i+1}=(I-\eta\nabla_a G(t,x,a_i)^\top G(t,x,a_i))\mu_i+O(\|\mu_i\|^2).
\end{equation}
The property of $\nabla_a G(t,x,a^*)$ is having rank $b$. Thus it can be decomposed to be $U\Sigma_bV^\top$. Similarly, the term $\Psi=\nabla_a G(t,x,a_i)^\top G(t,x,a_i)$ can be decomposed to be $V\Sigma^2_bV^\top$, where $\Sigma_b=diag([\sigma_1,\cdots,\sigma_b,0,\cdots,0]^\top)$ and $V$ is orthogonal eigenvector matrix of $\Psi$ with eigenvalue $\lambda_i=\sigma^2_i(i=1,\cdots,b)$. Let $I_b=diag([1_1,\cdots,1_b,0,\cdots,0]^\top)$ and $I_{m-b}=diag([0_1,\cdots,0_b,1,\cdots,1]^\top)$. The convergence is on $I_bV^\top \mu_i$ and $I_bU^\top G(t,x,a_i)$. 
\begin{equation}
0\le\delta^\top G(t,x,a_i)=(\mu_i)^\top\left(\Psi^2\right)\mu_i+O(\|\mu_i\|^3),
\end{equation}
and becomes zero if and only if $I_bV^\top\mu_i=0$, where $\Psi^2=\Psi\Psi$, and
\begin{equation}
0\ge \delta^\top\delta=G(t,x,a_i)^\top\left(\Pi^2\right)G(t,x,a_i)+ O(\|\mu_i\|^3),
\end{equation} 
and be zero if and only if $I_bU^\top G(t,x,a_i)=0$ where $\Pi=\nabla_a G(t,x,a_i)\nabla_a G(t,x,a_i)^\top$ and $\Pi^2=\Pi\Pi$. Thus
\begin{equation}
\delta^\top\delta\le\max_{i}\lambda^2_i\|I_bU^\top G(t,x,a_i)\|^2_2\quad\mbox{and}\quad \delta^\top G(t,x,a_i)\ge\min_{i}\lambda^2_i\|I_bV^\top\mu_i\|^2_2.
\end{equation}
It can imply that
\begin{equation}
\|G(t,x,a_{i+1}\|^2_2\le \left(1-\frac{min_i\lambda^4_i}{\max_i\lambda^4_i}\right)\|G(t,x,a_i)\|^2_2,
\end{equation}
of linear convergence. 
\end{proof}

\section{Illustrative Numerical Application}\label{Num_example}

In this section, the application of the proposed algorithms are given in detail for the respected example problem. The Jacobian matrix $\nabla_aF$ and a derivative vector $E$ are explained in detail. \\
For population dynamic and Lorenz's system example, the Jacobian matrix is static, whereas for activator-inhibitor example, the Jacobian matrix is updated for each iteration. \\
In the implementation of the proposed algorithms \ref{algorithm} and \ref{algorithm22}, data are generated by solving the related system of ODEs based on the defined parameters. The data are then added noise with normal distribution. However, in the implementation, the data are not filtered. The application of the algorithms \ref{algorithm_sNR} and \ref{algorithm22} is presented by using Python (the source code will be given via Github after the acceptance of the article).\\
Suppose solution of systems of ODEs (\ref{eq04}) is given as following
\begin{equation}
x^v_k=g_j(t,x_k,a),
\end{equation}
where $g_j(\cdot)$ is solution functions. Thus, the residue of $x^v_k$ to the $x^d_k$ can be presented as 
\begin{equation}
e_k=x^v_k-x^d_k.
\end{equation}
Hence, the comparison of the proposed methodologies to nonlinear least square (NLS) is 
conducted. Thus, to estimate the unknown parameters $a$, it can be presented in a constraints (NLS) optimization method as
\begin{equation}
\begin{aligned}
\min_{e} \quad & e^2_k,\\
\textrm{subject to} \quad & b_p(a)=\alpha_p,\quad p=1,\cdots, u,\\
  &c_q(a)\ge\beta_q  \quad q=1,\cdots,v,  \\
\end{aligned}
\end{equation}
where $b_p(a)=\alpha_p$ and $c_q(a)\ge\beta_q$ are the constraints ($u$ number equality constraints, and $v$ number of inequality constraints) to be satisfied for feasible set of solution.
Applying Lagrange multipliers to the minimization problem, and having derivative to the respective unknown parameters $a$, the solution can be obtained. \\
In this study, the proposed methodologies are compared with the well-known NLS method, following the approach outlined in this
https://www.mathworks.com, 
where the fitting is performed using the numerical solution of the given ODEs and the provided data set. While, the proposed methodologies and algorithms \ref{algorithm} and \ref{algorithm2} are fitted to the ODEs functions and the given data set. This is a clear contribution of this article. 
\subsection{Fitting a population system}
Suppose a given data set $(t_d,x^d_k),\; d=1,\cdots,N\;k=1,2$
is to be fitted to the following system of equations
\begin{equation}
    \begin{split}
    \frac{dx_1}{dt}=&f_1(t,x,a_l)=a_1x_1-a_2x_1x_2,\\
    \frac{dx_2}{dx}=&f_2(t,x,a_l)=a_3x_2+a_4x_1x_2-a_5x^2_2.
    \end{split}\label{fe01}
\end{equation}
The parameters $a_l,\;l=1,\cdots, 5$ are going to be estimated from the available generated data. From this example, it clearly can be observed that the equation (\ref{fe01}) is linear in parameters and nonlinear in states.  \\
\subsubsection{Implementation NR method }
Hence, from those available data $(t_d,x^d_k)$ and initial guess parameters $a_0$, the error vector $E$ can be presented as
\begin{equation}
E(t,x,a_i)=\begin{bmatrix}
f_1(t,x^d,a_i)-x'_1(t_d)\\
f_2(t,x^d,a_i)-x'_2(t_d)\\
\vdots
\end{bmatrix},\quad d=1,\cdots, N-1.
\end{equation}
It is also similarly for Jacobian's matrix, which is presented as the following
\begin{equation}\resizebox{0.9\hsize}{!}{$
\nabla_aF(t,x,a_i)=\begin{bmatrix}
\left.\frac{\partial f_1}{\partial a_1}\right|_{a_i,x^d}&\left.\frac{\partial f_1}{\partial a_2}\right|_{a_i,x^d}&\left.\frac{\partial f_1}{\partial a_3}\right|_{a_i,x^d}&\left.\frac{\partial f_1}{\partial a_4}\right|_{a_i,x^d}&\left.\frac{\partial f_1}{\partial a_5}\right|_{a_i,x^d}\\
\left.\frac{\partial f_2}{\partial a_1}\right|_{a_i,x^d}&\left.\frac{\partial f_2}{\partial a_2}\right|_{a_i,x^d}&\left.\frac{\partial f_2}{\partial a_3}\right|_{a_i,x^d}&\left.\frac{\partial f_2}{\partial a_4}\right|_{a_i,x^d}&\left.\frac{\partial f_2}{\partial a_3}\right|_{a_i,x^d}\\
\vdots&\vdots&\vdots&\vdots&\vdots
\end{bmatrix},\quad  d=1,\cdots, N-1$},
\end{equation}
where
\begin{equation*}\resizebox{0.9\hsize}{!}{$
\left.\frac{\partial f_1}{\partial a_1}\right|_{a_i,x^d}=x^d_1,\quad \left.\frac{\partial f_1}{\partial a_2}\right|_{a_i,x^d}=-x^d_1x^d_2,\quad
\left.\frac{\partial f_1}{\partial a_3}\right|_{a_i,x^d}=0,\quad \left.\frac{\partial f_1}{\partial a_4}\right|_{a_i,x^d}=0,\quad
\left.\frac{\partial f_1}{\partial a_5}\right|_{a_i,x^d}=0$},
\end{equation*}
\begin{equation*}\resizebox{0.9\hsize}{!}{$
\left.\frac{\partial f_2}{\partial a_1}\right|_{a_i,x^d}=0,\quad \left.\frac{\partial f_2}{\partial a_2}\right|_{a_i,x^d}=0,\quad\left.\frac{\partial f_2}{\partial a_3}\right|_{a_i,x^d}=x^d_2,\quad \left.\frac{\partial f_2}{\partial a_4}\right|_{a_i,x^d}=x^d_1x^d_2,\quad \left.\frac{\partial f_2}{\partial a_3}\right|_{a_i,x^d}=-(x^d_2)^2$}.
\end{equation*}
Thus vector $E(t,x,a_i)$ and matrix $\nabla_a F(t,x,a_i)$ have $(n(N-1))$-rows. Then, the unknown vector 
$a_l,\;l=1,\cdots,5 $
can be solved iteratively following equation (\ref{NR03}). 
From the updated parameters $a_i$, the vector $E(t,x,a_i)$ is updated. The Jacobian $\nabla_a F(\cdot)$ matrix is constant. Then, find new vector $a_i$, continue this cycle until $a_i$ converts. The searching method is presented in the algorithm \ref{algorithm}. 
\subsubsection{Implementation GD method}
Applying GD method to this problem is by taking
\begin{equation}
G(t,x,a_i)=E(t,x,a_i),
\end{equation}
and
\begin{equation}
\nabla_aG(t,a_i)=\nabla_aF(t,x,a_i).
\end{equation}
Thus $a_i$ is updated following equation (\ref{GD02}) and algorithm \ref{algorithm2} until criteria is fulfilled.

\subsubsection{Numerical implementation}
The methodology is implemented numerically to obtain unknown parameters of the problem in equation (\ref{fe01}). The implementation follows the proposed algorithms \ref{algorithm} and \ref{algorithm2}.
For this study, data $(t_d,x^d_k)$ are generated by using parameters $a=[
10\;5\:3\;1\;3
]^\top$.
In the implementation of the proposed algorithms \ref{algorithm} and \ref{algorithm2}, the initial guesses of parameters are selected as shown in Table  \ref{TableIni01}. \\
Both algorithms upon the various initial guesses (shown in Table \ref{TableIni01}) end up at the optimal parameters as 
\begin{equation}
\bar{a}=\begin{bmatrix}10.0646&5.0314&3.1650&1.0222&3.1166\end{bmatrix}^\top,
\end{equation}
which give an error to original parameters as
\begin{equation}
e_{\bar{a}}=|\bar{a}-a|=\begin{bmatrix}
0.0646&0.0314&0.1650&0.0222&0.1166
\end{bmatrix}^\top,
\end{equation}
which $e_{\bar{a}}$ is less than 17\%. \\
The fitting for the system of ODEs from the obtained parameters and the generated data is plotted in figure \ref{fit_firt_example}. It is shown that the generated data (tri-right) are agreed to the estimated  data (solid line). Figure \ref{fit_firt_example}.a and \ref{fit_firt_example}.c is the fitting for states $x_1$ and $x_2$ of generated data from original parameters $a$ and estimated parameters $\bar{a}$ by using NR method and using GD method, respectively. The states for estimated parameters are obtained by solving equation (\ref{fe01}) numerically by using the obtained optimal parameters. On plot depicted in figure \ref{fit_firt_example}.b by using NR method and \ref{fit_firt_example}.d using GD method, it is fitted for $f_j(t,x,a)$ (tri-right), and $f_j(t,x,\bar{a})$ (solid line). 

\begin{figure}
\centering
\begin{subfigure}[h!]{0.5\textwidth}
\includegraphics[width=\textwidth]{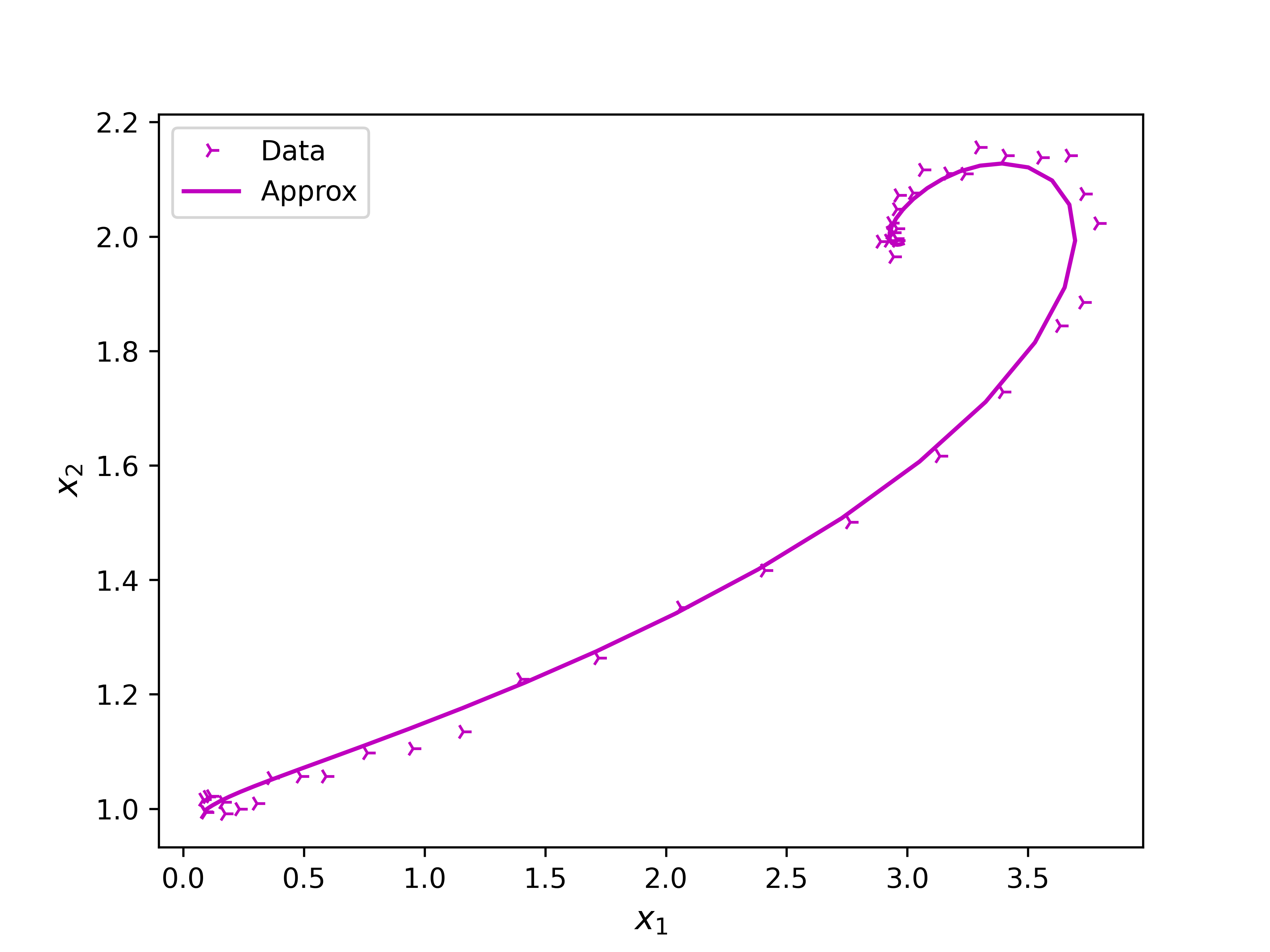}
\caption{fitting to data using NR method}
\end{subfigure}~
\begin{subfigure}[h!]{0.5\textwidth}
\includegraphics[width=\textwidth]{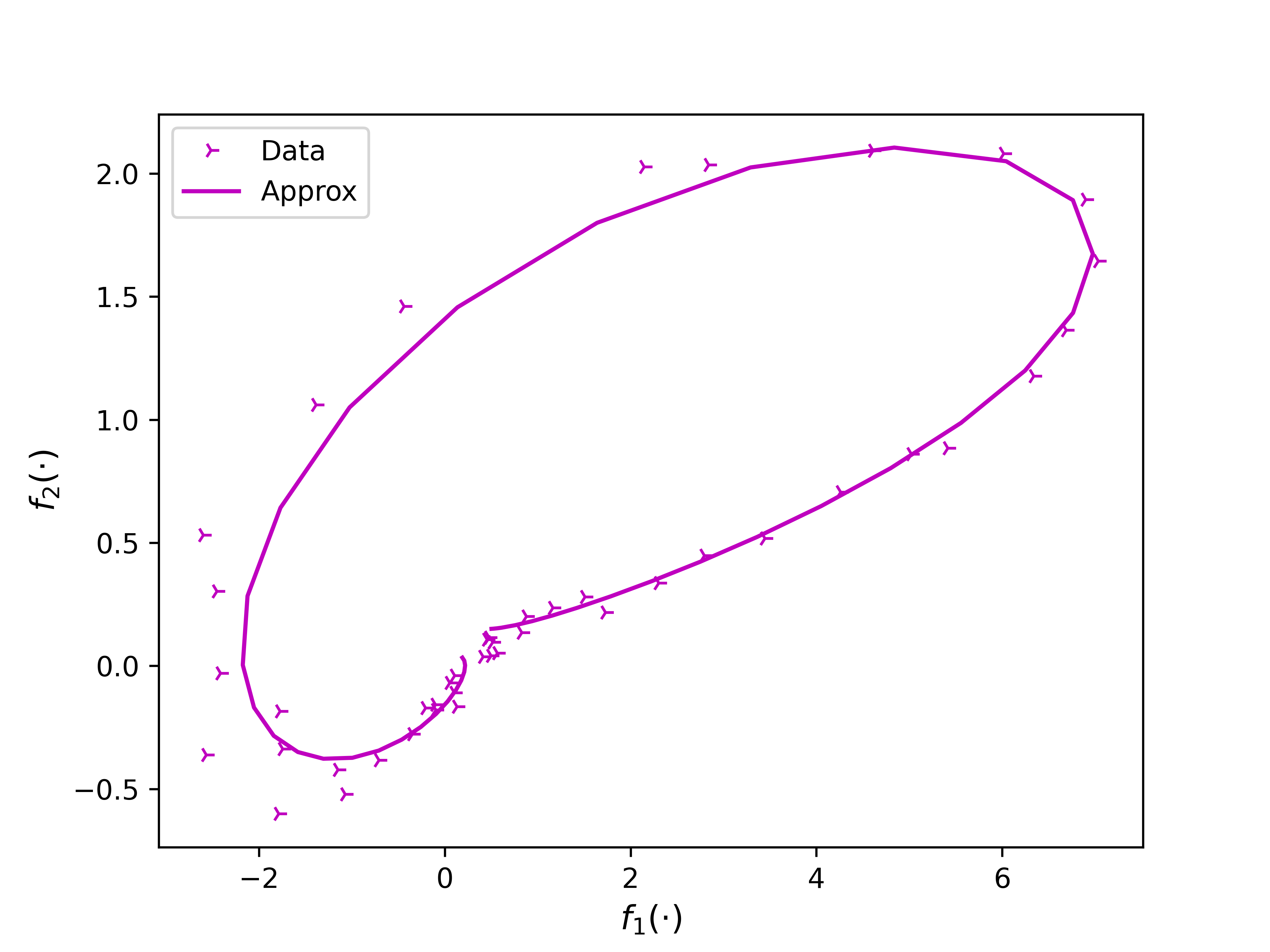}
\caption{fitting to derivative functions using NR method}
\end{subfigure}\\
\begin{subfigure}[h!]{0.5\textwidth}
\includegraphics[width=\textwidth]{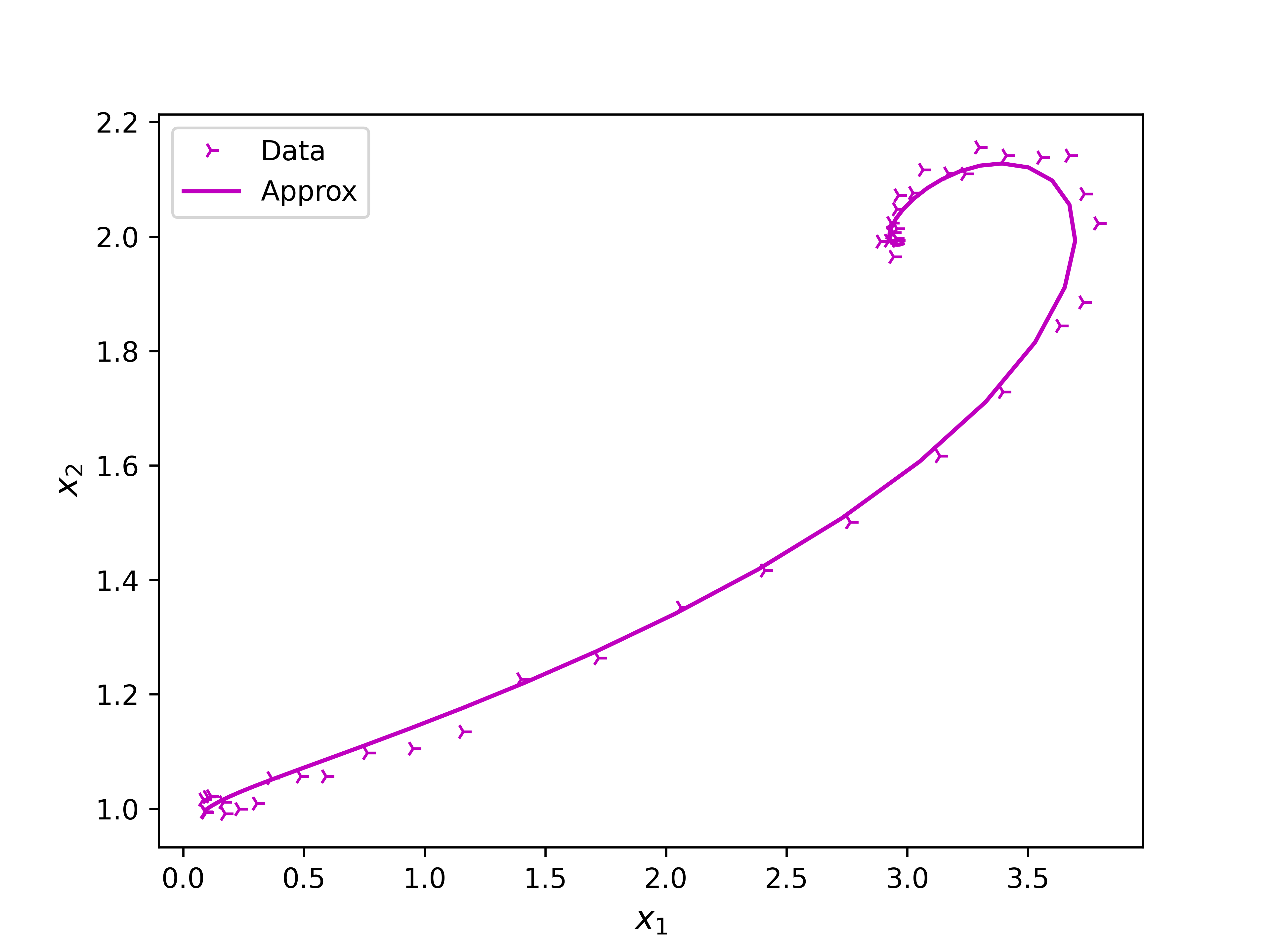}
\caption{fitting to data using GD method}
\end{subfigure}~
\begin{subfigure}[h!]{0.5\textwidth}
\includegraphics[width=\textwidth]{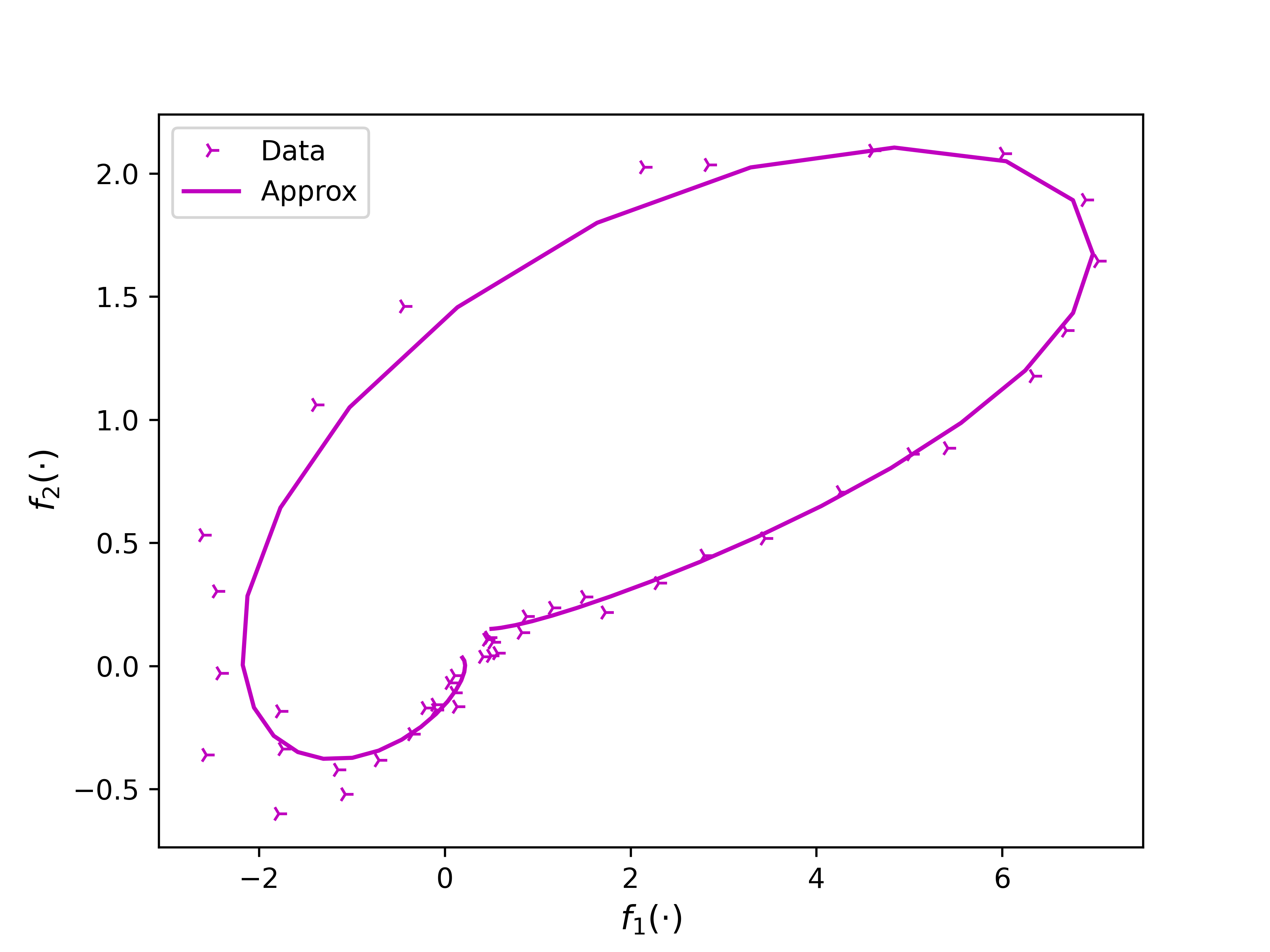}
\caption{fitting to derivative functions using GD method}
\end{subfigure}
\caption{population dynamic Example: scattered plots are original data, solid lines are estimated data}\label{fit_firt_example}
\end{figure}
\begin{table}\centering
\caption{Initial and convergent parameters and iteration - CPU time of a population system example}\label{TableIni01}\footnotesize
\begin{tabular}{|c|c|c|c|c|c|}
\hline
\rotatebox[origin=c]{90}{Initial guess}&$\begin{bmatrix}
0.1\\0.1\\1.2\\1.3\\0.2
\end{bmatrix}$&$\begin{bmatrix}
1\\ -1\\ 2\\ 0\\1
\end{bmatrix}$&$\begin{bmatrix}
-10\\ -10\\ 2\\ -3\\1
\end{bmatrix}$&$\begin{bmatrix}
0\\0\\0\\0\\0
\end{bmatrix}$&$\begin{bmatrix}
100\\ -100\\ -100\\ 20\\-30
\end{bmatrix}$\\
\hline
\multirow{2}{*}{NR}&2-iter &2-iter&2-iter&2-iter&2-iter\\
& $\approx 0$ s&0.01046 s& $\approx 0$  s& $\approx 0$s& 0.0017  s\\
\hline
\multirow{2}{*}{GD}& 5495-iter &2557-iter&4794-iter s &3293-iter&5784-iter\\
&0.8724 s&0.4074 s& 0.7597& 0.5260 s& 0.9062 s\\
\hline
\end{tabular}
\end{table}

\begin{table}\centering
\caption{Error analysis and fitting coefficient for a population system example using both NR and GD methods}\label{EA_first_example_proposed_method}
\begin{tabular}{|c|c|c|c|c|c|}
\hline
variable&bias&MAPE&MAE&RMSE&R$^2$\\ \hline
$x_1$&0.0158&0.0331&0.0271&0.0384&0.9992\\
$x_2$&0.0031&0.0119&0.0176&0.0207&0.9980\\
$f_1(\cdot)$&-0.1028&0.4263&0.2088&0.2755&0.9898\\
$f_2(\cdot)$&-0.0298&0.4500&0.0681&0.0937&0.9851\\
\hline
\end{tabular}
\end{table}
Even though high in computational cost, NR method converges to the optimal searching parameters in a very fast CPU time and relatively 2 iterations. However, using GD method, the CPU time is less than 1 seconds and more than 2500 iterations to converge to the optimal parameters. These CPU time and iteration number are reported in Table \ref{TableIni01}, for this example. 

The error analysis for bias, MAPE, MAE, RMSE and coefficient of determination (R$^2$) is tabulated in Table \ref{EA_first_example_proposed_method}. For states estimation, it is found that bias, MAPE, MAE, RMSE are close to 0 for both states $x_1, x_2$. The fitting coefficient $R^2$ for state $x_1$ slightly better to state $x_2$, which are both approach 1. Furthermore, state $x_2$ has lower bias, MAPE, MAE and RMSE. Similarly, the fitting to the derivative functions $f_1(\cdot), f_2(\cdot)$ in term of fitting coefficient $R^2$, $f_1(\cdot)$ is more closer to 1 than $f_2(\cdot)$. In the other side, $f_1(\cdot)$ has slightly higher in MAE and RMSE and lesser in bias compared to $f_2(\cdot)$. MAPE analysis for $f_2(\cdot)$ has little bit higher than $f_1(\cdot)$. However, the obtained optimal parameters are acceptable which is giving error analyses closed to 0 and fitting coefficient closed to 1.    
It can be concluded that the both methods can estimate very well the ODEs parameters of the population dynamic example. 
\subsubsection{Comparison to nonlinear least square}
From those implementation of constrained NLS method with bounded parameters, 
\begin{equation}
0\le a\le 11,
\end{equation} with the same initial guessed parameters, 
the estimated parameters are obtained 
\begin{equation}
\bar{a}=\begin{bmatrix}
8.0394&3.9912&2.4729&1.1738&2.9589
\end{bmatrix}.
\end{equation}
From these parameters, The fitting is shown in figure \ref{NLS_fitting_firt_example}. It is clearly observed that the parameters produce a shift on fitting to data points. 
\begin{figure}[h!]
\centering
\begin{subfigure}[h!]{0.5\textwidth}
\includegraphics[width=\textwidth]{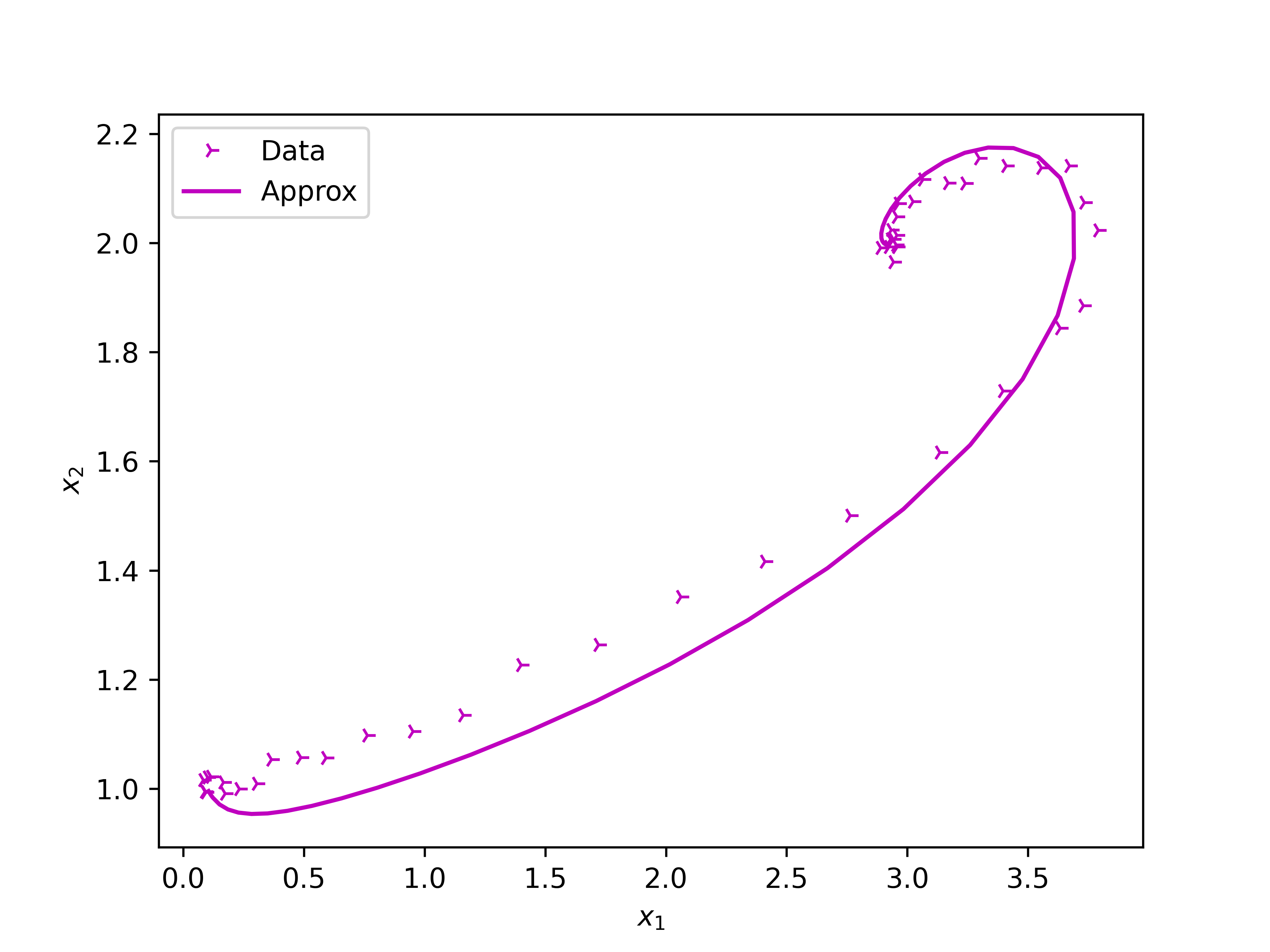}
\caption{NLS fitting to data}
\end{subfigure}~
\begin{subfigure}[h!]{0.5\textwidth}
\includegraphics[width=\textwidth]{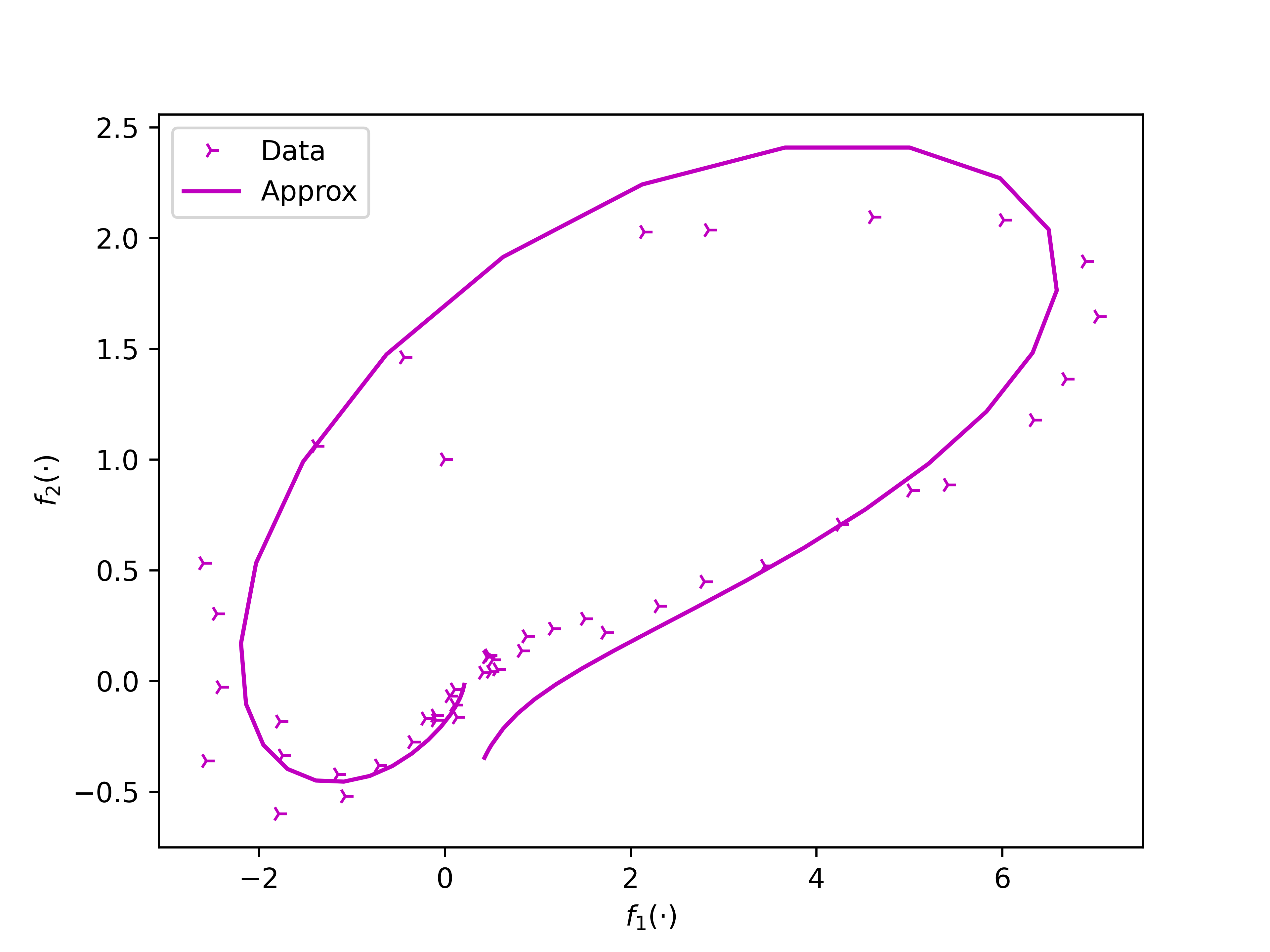}
\caption{NLS fitting to derivative functions}
\end{subfigure}
\caption{first example NLS fitting: dotted is data, continuous line is approximated data}\label{NLS_fitting_firt_example}
\end{figure}
\begin{table}\centering
\caption{Error analysis for population dynamic example using constrained NLS method}\label{EA_first_example_NLS}
\begin{tabular}{|c|c|c|c|c|c|}
\hline
variable&bias&MAPE&MAE&RMSE&R$^2$\\ \hline
$x_1$&0.0769&0.0760&0.1139&0.1741&0.9830\\
$x_2$&0.0537&0.0525&0.0764&0.1003&0.9535\\
$f_1(\cdot)$&-0.0727&0.6252&0.5815&0.8566&0.9010\\
$f_2(\cdot)$&0.0064&1.7939&0.2709&0.3492&0.7928\\
\hline
\end{tabular}
\end{table}\\
The fitting coefficient of NLS method base on error is tabulated in Table \ref{EA_first_example_NLS}. If the fitting coefficient of NLS method compares to the proposed method,  the parameters bias, MAPE, MAE, and RMSE are closer to 0 for the proposed methodology and algorithm \ref{algorithm}, except for bias of derivative function 1. However, coefficients of determination R$^2$ for proposed method are closer to 1. This means that the proposed methods and algorithms \ref{algorithm} and \ref{algorithm2} is superior to NLS method. 

\subsection{Fitting a Lorenz's system}
Suppose a given data set $(t_d,x^d_k),\; d=1,\cdots, N \; k=1,2,3$ is chaotic physical data of Lorenz's system,  
\begin{equation}
\begin{split}
\frac{dx_1}{dt}=&f_1(t,x,a_l)=a_1(x_2-x_1),\\
\frac{dx_2}{dt}=&f_2(t,x,a_l)=x_1(a_2-x_3)-x_2,\\
\frac{dx_3}{dt}=&f_3(t,x,a_l)=x_1x_2-a_3 x_3.
\end{split}\label{fe02}
\end{equation}
where $a_l,\; l=1,2,3$ are unknown parameters to be estimated based on the available data set using the proposed algorithms \ref{algorithm} and \ref{algorithm2}. \\
\subsubsection{Implementation NR method}
For this parameters searching, the derivative vector $E(t,x,a_i)$ can be presented as
\begin{equation}
E(t,x,a_i)=\begin{bmatrix}
f_1(t_d,x^d,a_i)-x'_1(t_d)\\
f_2(t_d,x^d,a_i)-x'_2(t_d)\\
f_3(t_d,x^d,a_i)-x'_3(t_d)\\
\vdots
\end{bmatrix},\quad d=1,\cdots, N-1.
\end{equation}\\
Then, the Jacobian matrix is the derivative of the function with respect to the unknown parameters. It can presented as
\begin{equation}
\nabla_a F(t,x,a_i)=\begin{bmatrix}
x^d_2-x^d_1&0&0\\
0&x^d_1&0\\
0&0&-x^d_3\\
\vdots&\vdots&\vdots
\end{bmatrix},\quad  d=1,\cdots, N-1.
\end{equation}
The update for searching Lorenz's parameters is following equation (\ref{NR03}) and algorithm \ref{algorithm}. Thus, the searching parameters are terminated if the condition in the algorithm is satisfied, in another word if the parameters are convergence. Clearly, the last updated parameters are the optimal parameters, which are the best for fitting the available data to the defined system of ODEs.

\subsubsection{Implementation GD method}
Similarly to the population dynamic example, it is by taking $G(t,x,a_i)=E(t,x,a_i)$ and $\nabla_a G(t,x,a_i)=\nabla_a F(t,x,a_i)$ of this example. Thus, the searching parameters $a_i$ are updated following equation (\ref{GD02}) and algorithm \ref{algorithm2}.
\subsubsection{Numerical implementation}
For the Lorenz's system of ODEs in this example, the data are generated by using the following parameters $a=[
10\;28\;8/3]^\top$ of the equation (\ref{fe02}) which is solved numerically from the initial condition $x_0=[0.1\;1\;5]^\top$. The solution is added normal distribution random variable to noise the states with variance 2. These data are used to estimate the system of ODEs parameters by using the proposed approaches and algorithms \ref{algorithm} and \ref{algorithm2}.\\
In this example, the implementation of the proposed methodologies and algorithms \ref{algorithm} and \ref{algorithm2} are needed to start with initial guessed parameters as shown in Table \ref{TableIni02}. 
As mentioned in population dynamic example, it can be selected freely. NR method converges to optimal searching parameters very fast which is less than 2 ms and only 1 iteration, which gives
$\bar{a}=[9.4746\;28.0349\;2.8675]^\top$.
While, for applying GD method, the optimal searching parameters convert also very fast which is less than 1 seconds and 155 iterations. However, by using NR method, the obtained optimal parameters are ended-up to same point from different initial guess. Meanwhile, using GD method the searched optimal parameters may end-up in a local minima as seen in Table \ref{TableIni02} for 5 initial guess trials. \\
\begin{table}\centering
\caption{Initial and convergent parameters and iteration - CPU time of Lorenz's system example}\label{TableIni02}\footnotesize
\begin{tabular}{|c|c|c|c|c|c|}
\hline
\rotatebox[origin=c]{90}{Initial guess}&$\begin{bmatrix}
15\\1\\-10
\end{bmatrix}$&$\begin{bmatrix}
0\\ 10\\ -10
\end{bmatrix}$&$\begin{bmatrix}
30\\ -10\\ 10
\end{bmatrix}$&$\begin{bmatrix}
-3\\ 1\\ 0
\end{bmatrix}$&$\begin{bmatrix}
0\\0\\0
\end{bmatrix}$\\
\hline
\multirow{3}{*}{NR}&1-iter &1-iter&1-iter&1-iter&2-iter\\
 &0.0020 s&$\approx 0$ &$\approx 0$ &$\approx 0$ &$\approx 0$  \\
&
 \multicolumn{5}{c|}{ 
$\begin{bmatrix}
9.47459164&28.03494241&2.86747935
\end{bmatrix}^\top $}\\
\hline
\multirow{3}{*}{GD}& 25-iter&23-iter&115-iter&155-iter&155-iter\\
&0.84978 s &0.8684 s&0.922 s &0.9347 s&0.9102 s\\
&$\begin{bmatrix}
7.95711877\\27.97766457\\ 2.81835555
\end{bmatrix}$&$\begin{bmatrix}
8.91327758\\27.97766457\\ 2.81213018
\end{bmatrix}$&$\begin{bmatrix}
9.75546914\\ 27.97766457\\  2.80634861
\end{bmatrix}$&$\begin{bmatrix}
9.67565866\\ 27.97766457 \\ 2.80752829
\end{bmatrix}$&
$\begin{bmatrix}
9.68142166\\27.97766457\\ 2.80741994
\end{bmatrix}$\\
\hline
\end{tabular}
\end{table}

\begin{table}\centering
\caption{Error analysis of using GD method for finding Lorenz's system parameters}\label{GD_error_analysis}
\begin{tabular}{|l|r|r|r|r|r|}
\hline
\multicolumn{1}{|c|}{$a_0$}& $\begin{bmatrix}
15\\1\\-10
\end{bmatrix}$&$\begin{bmatrix}
0\\ 10\\ -10
\end{bmatrix}$&$\begin{bmatrix}
30\\ -10\\ 10
\end{bmatrix}$&$\begin{bmatrix}
-3\\ 1\\ 0
\end{bmatrix}$&$\begin{bmatrix}
0\\0\\0
\end{bmatrix}$\\
\hline
bias&&&&&\\
\multicolumn{1}{|r|}{$x_1$}&0.0569&-0.2394&0.0369&0.0211&0.0227\\
\multicolumn{1}{|r|}{$x_2$}&0.1316&0.0654&0.1348&0.1380&0.1377\\
\multicolumn{1}{|r|}{$x_3$}&0.3502&1.1577&0.4240&0.4859&0.4814\\
\multicolumn{1}{|r|}{$f_1(\cdot)$}&0.7734&2.5661&0.9632&1.1215&1.1065\\
\multicolumn{1}{|r|}{$f_2(\cdot)$}&0.9704&2.8068&0.8525&0.8255&0.8256\\
\multicolumn{1}{|r|}{$f_3(\cdot)$}&0.3170&-1.7818&-0.1898&-0.4955&-0.4690\\
\hline
MAPE&&&&&\\
\multicolumn{1}{|r|}{$x_1$}&0.2495&0.3323&0.2095&0.1986&0.1990\\
\multicolumn{1}{|r|}{$x_2$}&0.2905&0.4830&0.2367&0.2214&0.2225\\
\multicolumn{1}{|r|}{$x_3$}&0.0937&0.1297&0.0704&0.0634&0.0636\\
\multicolumn{1}{|r|}{$f_1(\cdot)$}&4.1417&7.1081&4.7069&5.0499&5.0109\\
\multicolumn{1}{|r|}{$f_2(\cdot)$}&1.7077&3.2955&1.1756&1.0091&1.0153\\
\multicolumn{1}{|r|}{$f_3(\cdot)$}&3.0242&5.1359&2.4163&2.0935&2.1144\\
\hline
MAE&&&&&\\
\multicolumn{1}{|r|}{$x_1$}&1.3573&1.9689&1.1063&1.0578&1.0578\\
\multicolumn{1}{|r|}{$x_2$}&1.6019&2.7778&1.2401&1.1545&1.1604\\
\multicolumn{1}{|r|}{$x_3$}&1.9922&3.0877&1.3922&1.2128&1.2190\\
\multicolumn{1}{|r|}{$f_1(\cdot)$}&15.7965&22.6766&14.1684&13.6381&13.6546\\
\multicolumn{1}{|r|}{$f_2(\cdot)$}&18.3308&27.9114&12.9338&11.2522&11.2910\\
\multicolumn{1}{|r|}{$f_3(\cdot)$}&21.1255&36.1948&15.8557&14.3514&14.3723\\
\hline
RMSE&&&&&\\
\multicolumn{1}{|r|}{$x_1$}&1.6721&2.5351&1.3247&1.2562&1.2584\\
\multicolumn{1}{|r|}{$x_2$}&2.0420&3.6730&1.5058&1.3843&1.3889\\
\multicolumn{1}{|r|}{$x_3$}&2.4776&4.1983&1.7205&1.5220&1.5317\\
\multicolumn{1}{|r|}{$f_1(\cdot)$}&19.9965&30.2226&17.3583&17.0121&17.0124\\
\multicolumn{1}{|r|}{$f_2(\cdot)$}&23.3798&40.5785&15.9072&14.3062&14.3519\\
\multicolumn{1}{|r|}{$f_3(\cdot)$}&27.6067&51.3703&19.9318&18.6903&18.7016\\
\hline
$R^2$&&&&&\\
\multicolumn{1}{|r|}{$x_1$}&0.9228&0.8226&0.9516&0.9564&0.9563\\
\multicolumn{1}{|r|}{$x_2$}&0.9089&0.7053&0.9505&0.9581&0.9579\\
\multicolumn{1}{|r|}{$x_3$}&0.8895&0.6826&0.9467&0.9583&0.9578\\
\multicolumn{1}{|r|}{$f_1(\cdot)$}&0.7204&0.3613&0.7893&0.7976&0.7976\\
\multicolumn{1}{|r|}{$f_2(\cdot)$}&0.8233&0.4678&0.9182&0.9338&0.9334\\
\multicolumn{1}{|r|}{$f_3(\cdot)$}&0.8229&0.3868&0.9077&0.9188&0.9187\\
\hline
\end{tabular}
\end{table}
The optimal estimated parameters $\bar{a}$ int Table \ref{TableIni02} by using NR method are used to solve the Lorenz's system of ODEs \ref{fe02} numerically. The fitting of the estimated data to original data  is plotted in figure \ref{fit_Lorenz}.a and estimated derivative function to the estimated derivative function is shown in figure \ref{fit_Lorenz}.b. It can be observed that the data (tri-right) and the estimated output (solid line) are agreed. It also can be seen for error analysis in Table \ref{EA_second_example_NR_method}. The fitting coefficient for all states and derivative functions are above 90\% except for derivative functions $f_1(\cdot)$ and $f_3(\cdot)$, which are 77.62\% and 88.56\%, respectively. Other error analysis, such as MAE and RMSE are greater than 10 for all derivative functions. For bias and MAPE analysis, all states and derivative functions have lower than 6.  
It means that the NR algorithm \ref{algorithm} can estimate the parameters.\\
While, the estimated parameters (Table \ref{TableIni02}) by using GD method for each parameters based on their initial guesses are plotted in figure \ref{fit_Lorenz}.c to figure \ref{fit_Lorenz}.l, respectively. 
The fitting coefficient for all states and derivative functions based on initial guesses is reasonably good except for initial condition $[0\;10\;-10]^\top$ of derivative function which is 0.3613, 0.4678 and 0.3868, respectively. All convergent parameters from each initial guess, bias and MAPE have lower than 7.11, while for MAE and RMSE the errors are less than 51.40, as shown in Table \ref{GD_error_analysis}. \\
\begin{figure}
\centering
\begin{subfigure}[h!]{0.3\textwidth}
\includegraphics[width=\textwidth]{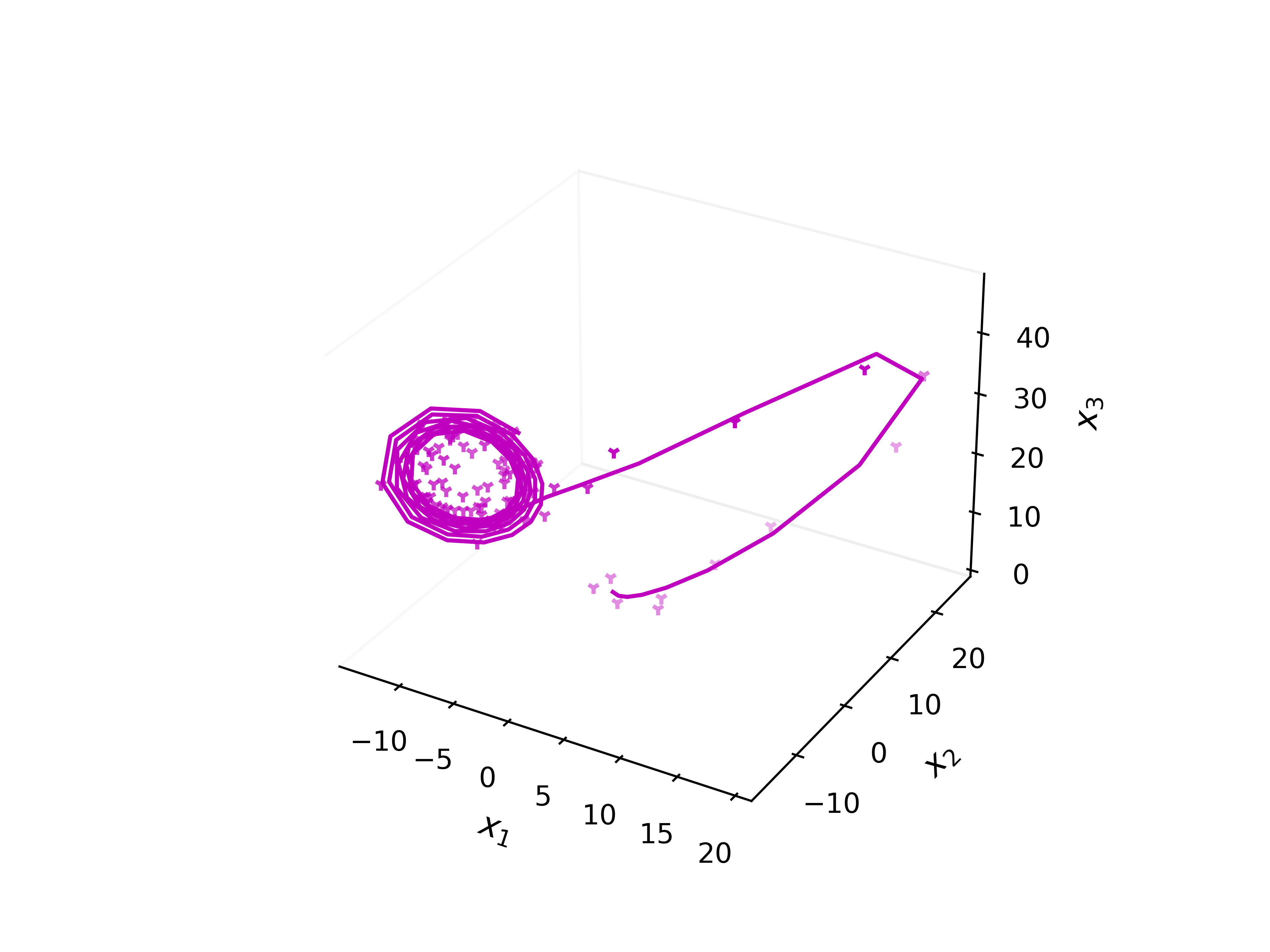}
\caption{fitting to data using NR method}
\end{subfigure}~
\begin{subfigure}[h!]{0.3\textwidth}
\includegraphics[width=\textwidth]{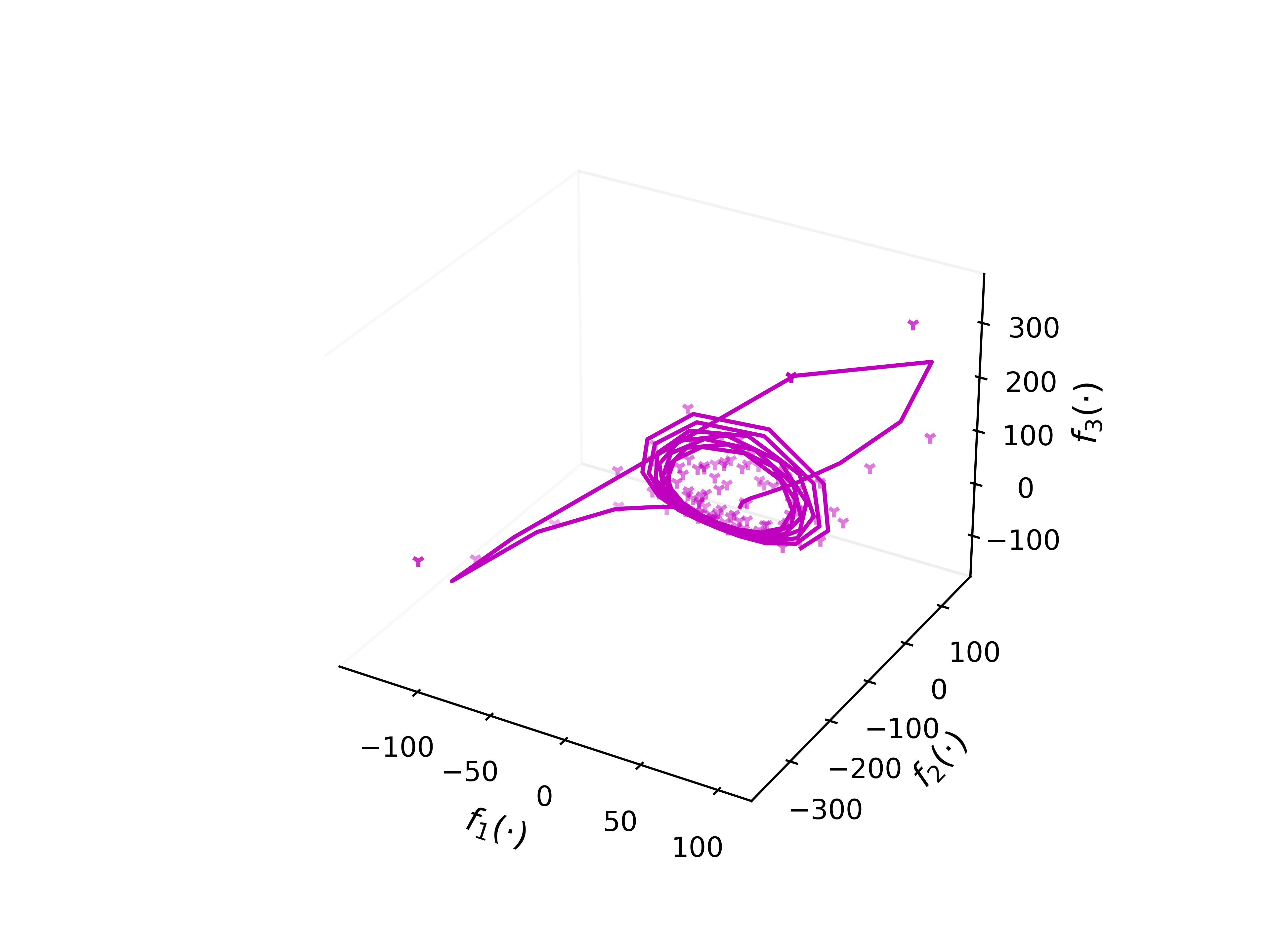}
\caption{fitting to derivative function using NR method}
\end{subfigure}
\begin{subfigure}[h!]{0.3\textwidth}
\includegraphics[width=\textwidth]{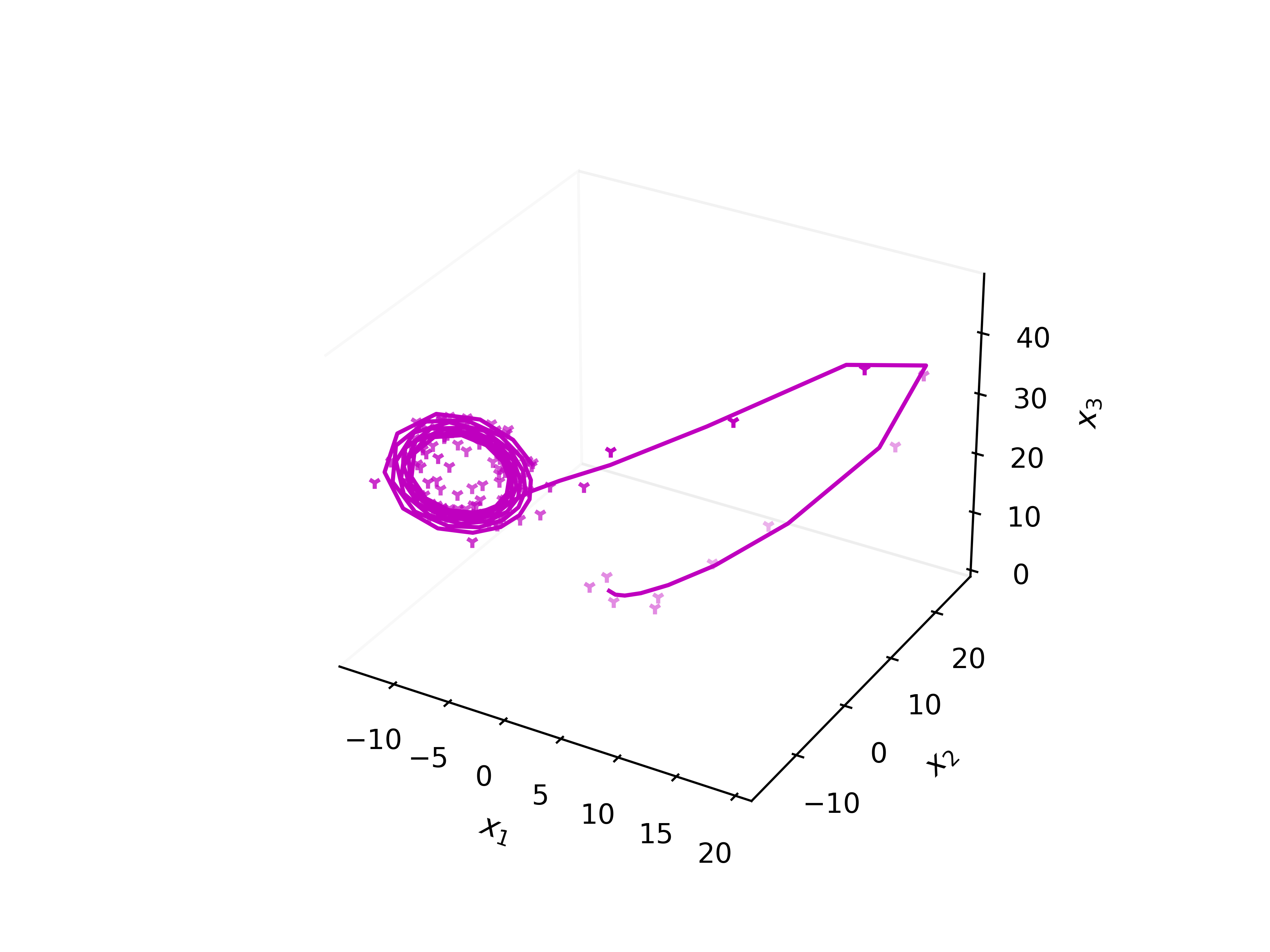}
\caption{fitting to data using GD method: first guess}
\end{subfigure}\\
\begin{subfigure}[h!]{0.3\textwidth}
\includegraphics[width=\textwidth]{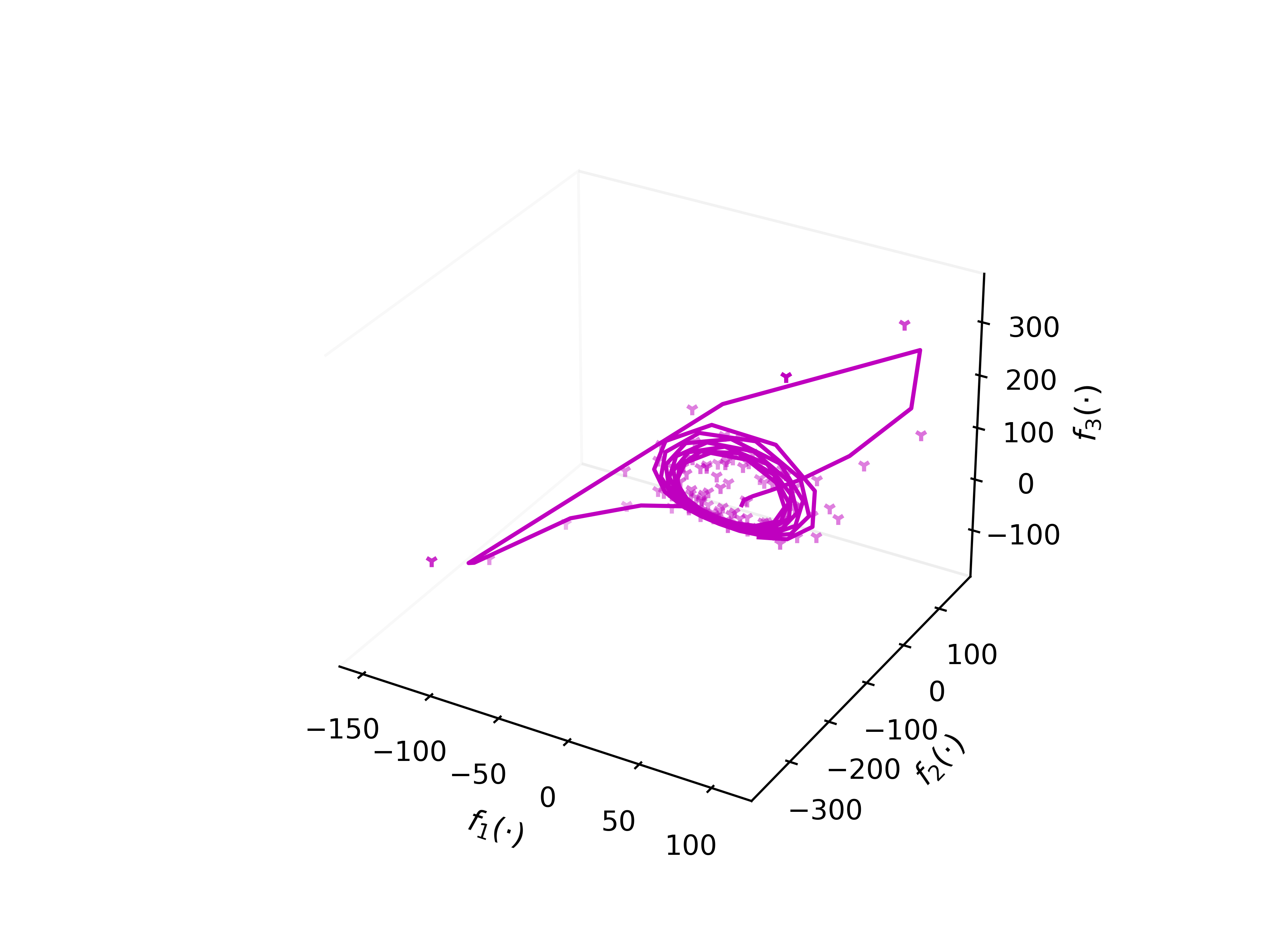}
\caption{fitting to derivative function using GD method: first guess}
\end{subfigure}~
\begin{subfigure}[h!]{0.3\textwidth}
\includegraphics[width=\textwidth]{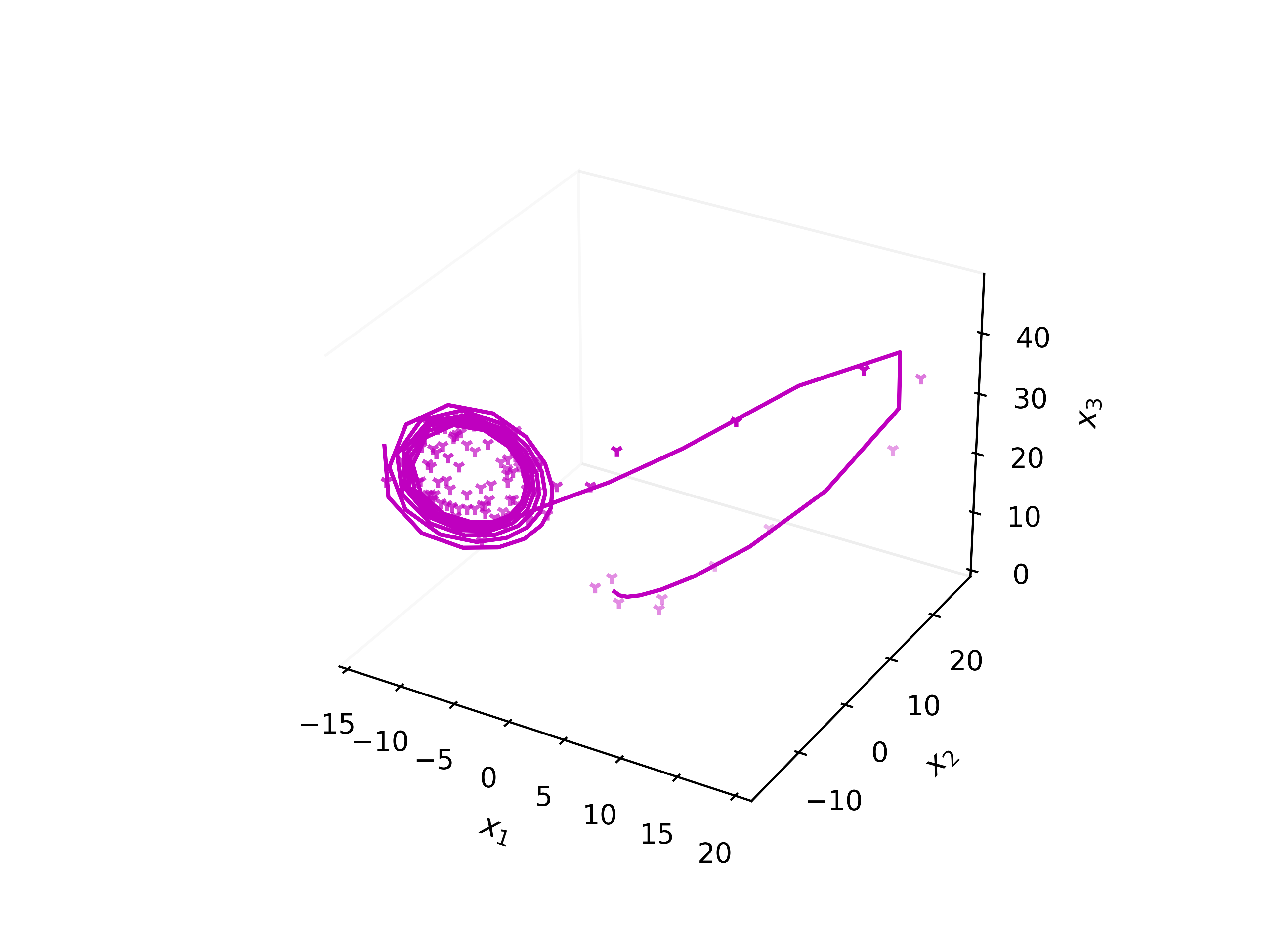}
\caption{fitting to data using GD method: second guess}
\end{subfigure}~
\begin{subfigure}[h!]{0.3\textwidth}
\includegraphics[width=\textwidth]{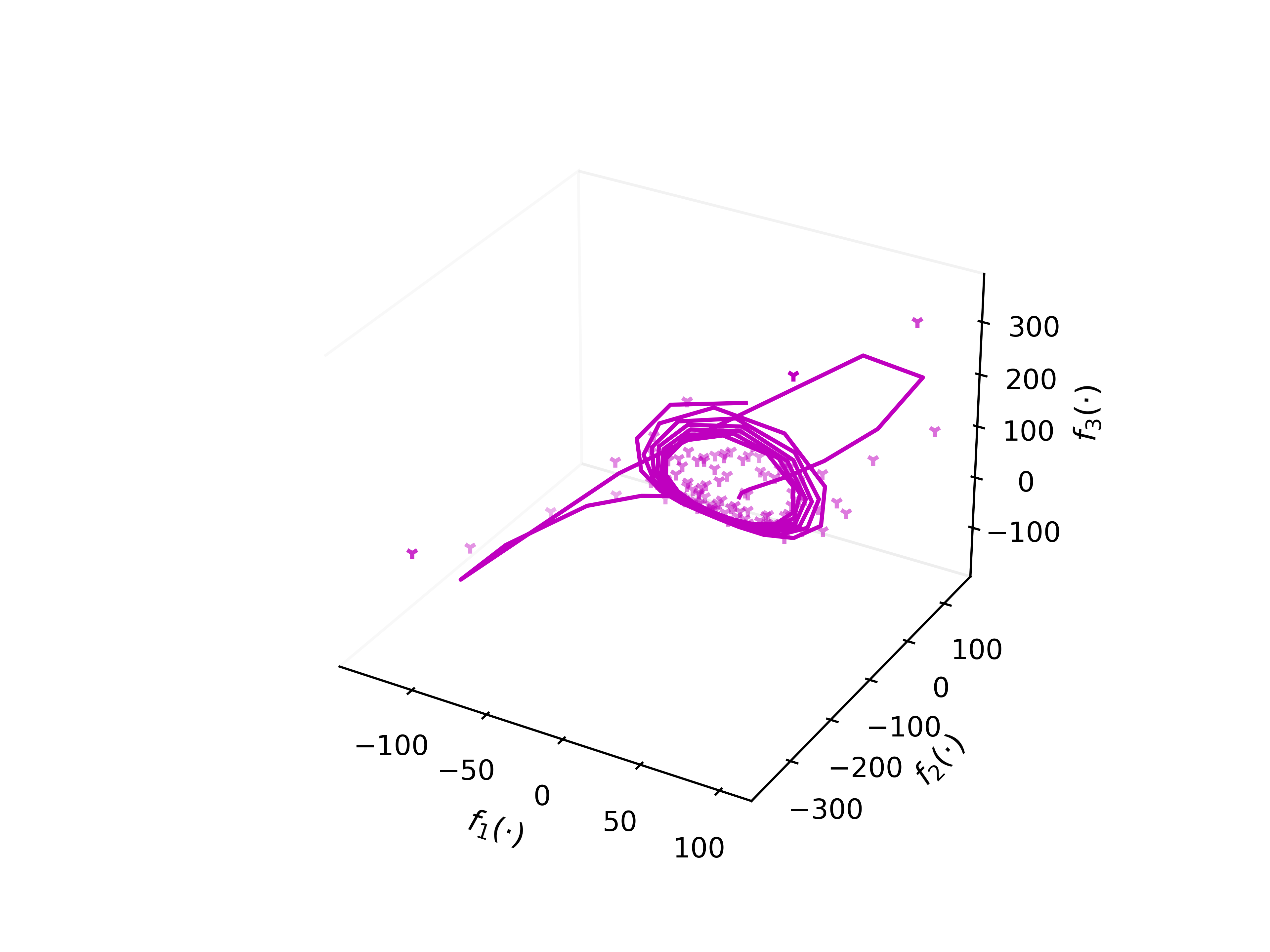}
\caption{fitting to derivative function using GD method: second guess}
\end{subfigure}\\
\begin{subfigure}[h!]{0.3\textwidth}
\includegraphics[width=\textwidth]{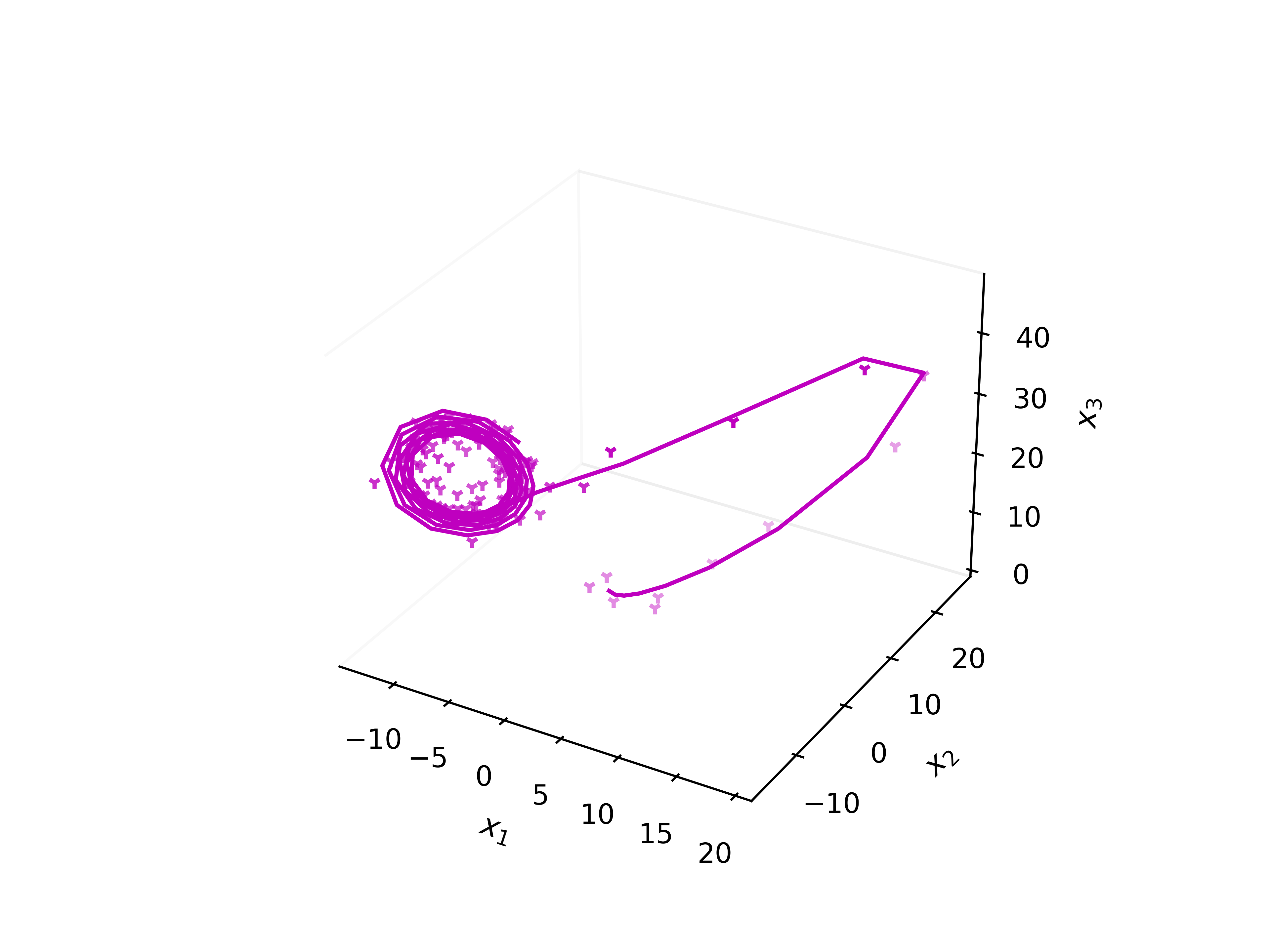}
\caption{fitting to data using GD method: third guess}
\end{subfigure}~
\begin{subfigure}[h!]{0.3\textwidth}
\includegraphics[width=\textwidth]{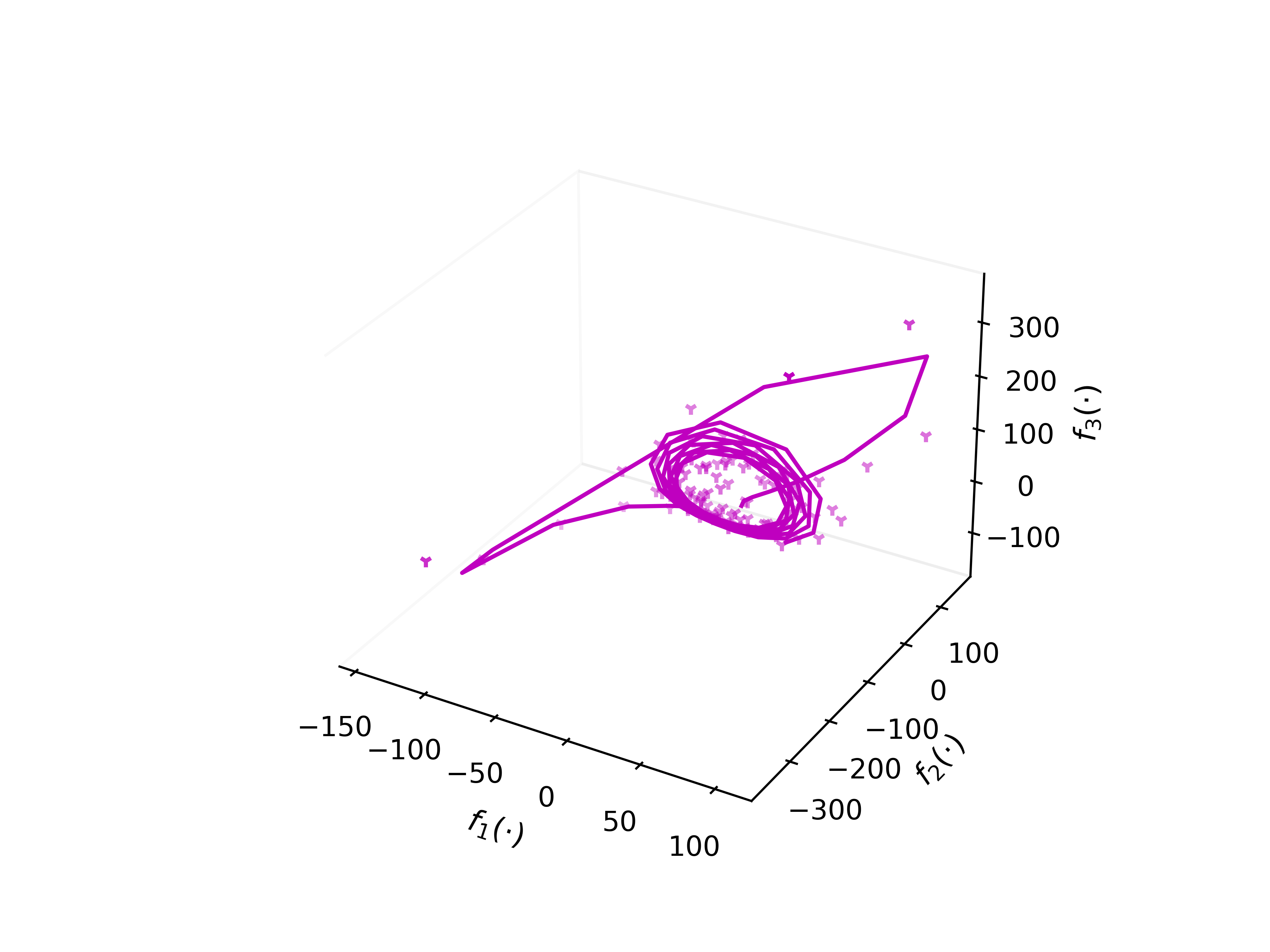}
\caption{fitting to derivative function using GD method: third guess}
\end{subfigure}~
\begin{subfigure}[h!]{0.3\textwidth}
\includegraphics[width=\textwidth]{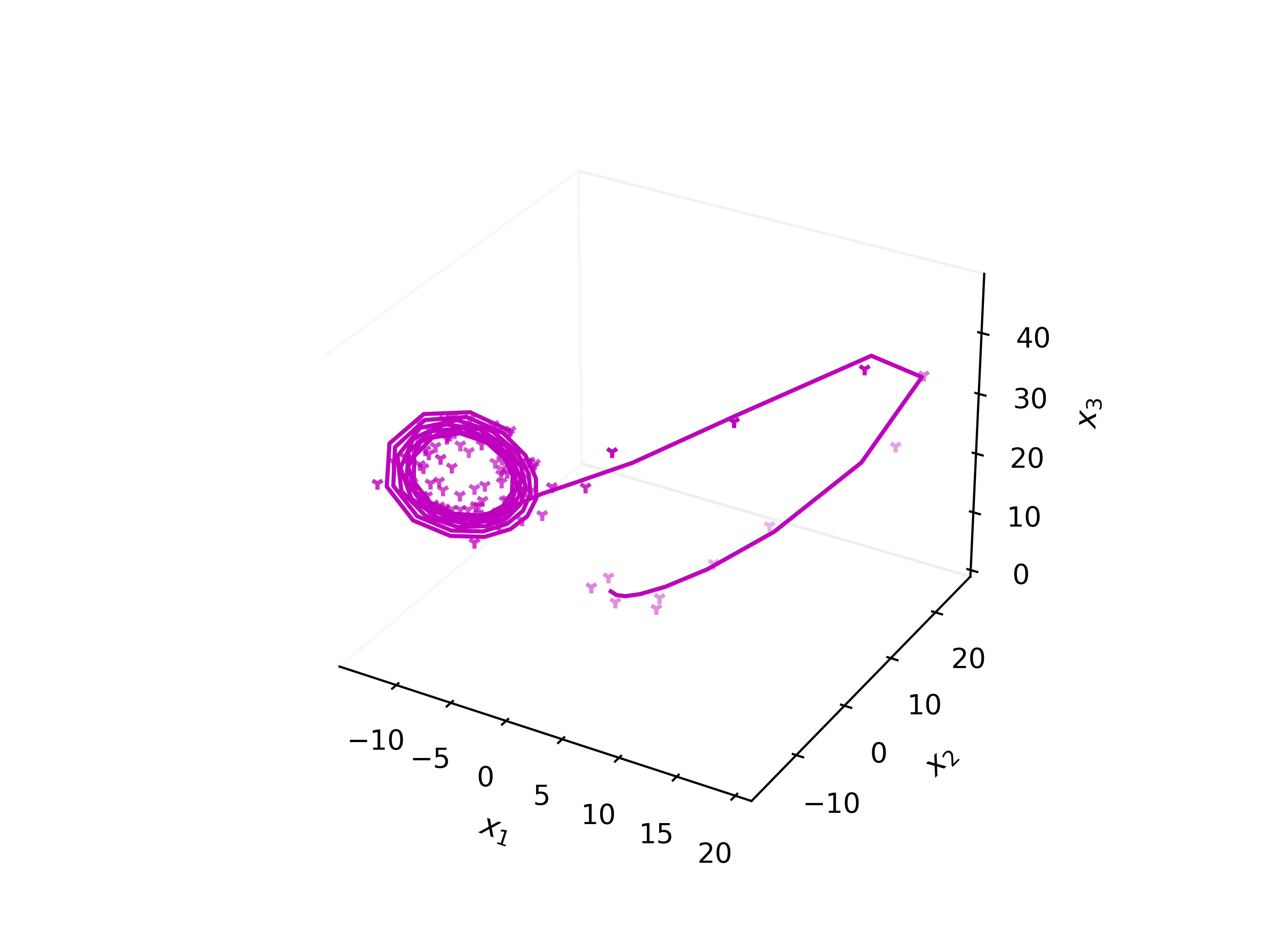}
\caption{fitting to data using GD method: fourth guess}
\end{subfigure}\\
\begin{subfigure}[h!]{0.3\textwidth}
\includegraphics[width=\textwidth]{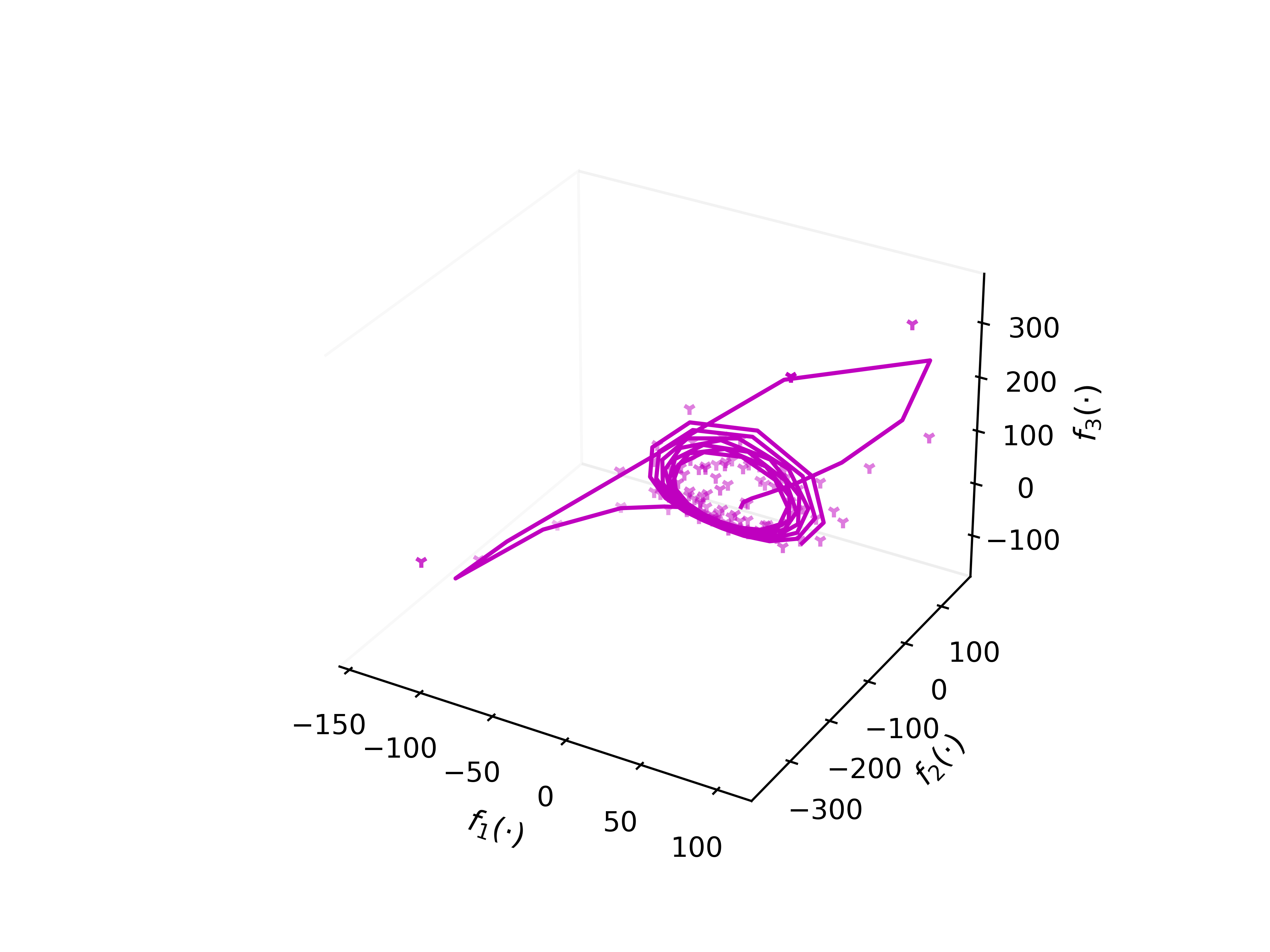}
\caption{fitting to derivative function using GD method: fourth guess}
\end{subfigure}~
\begin{subfigure}[h!]{0.3\textwidth}
\includegraphics[width=\textwidth]{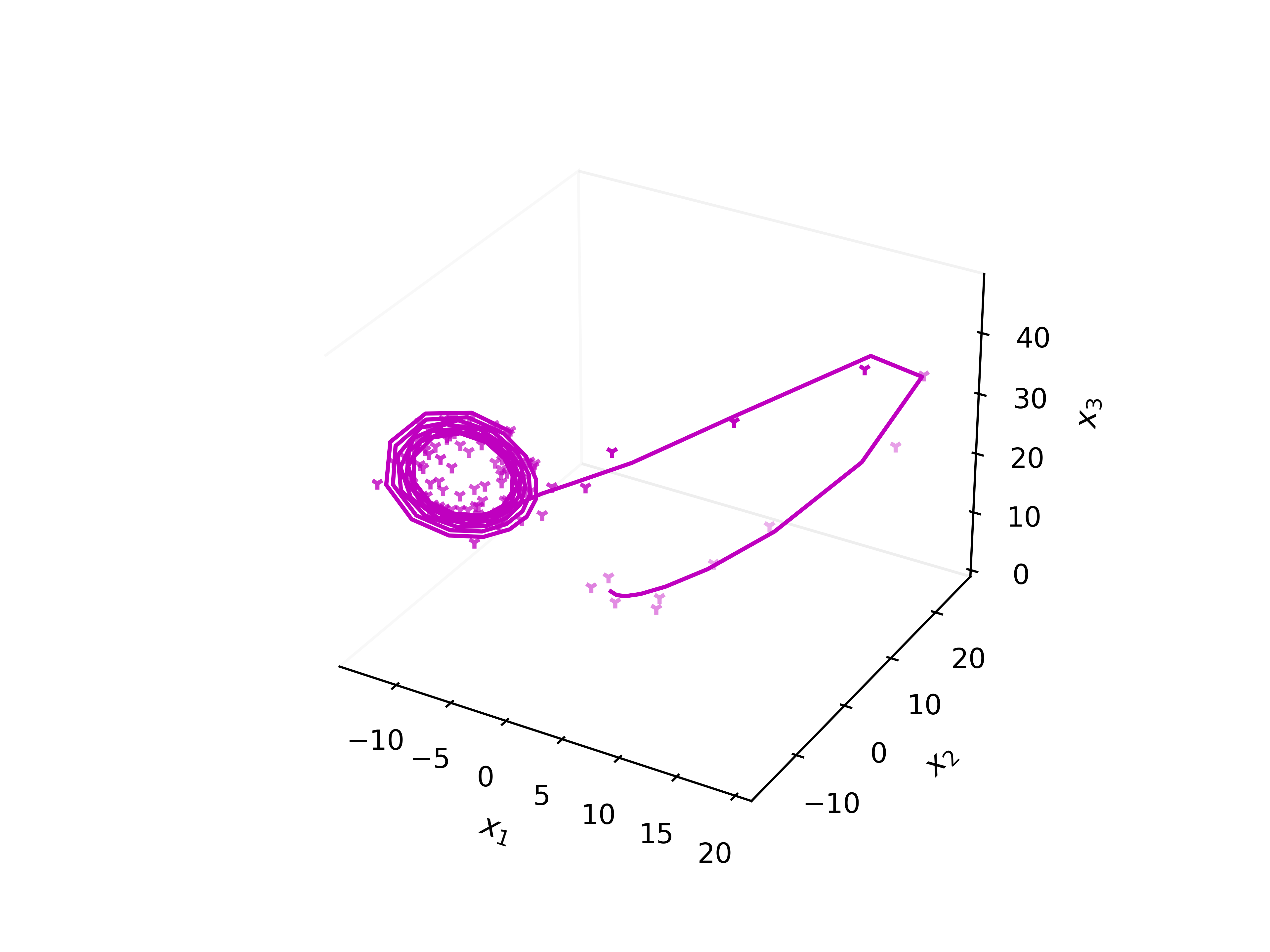}
\caption{fitting to data using GD method: fifth guess}
\end{subfigure}~
\begin{subfigure}[h!]{0.3\textwidth}
\includegraphics[width=\textwidth]{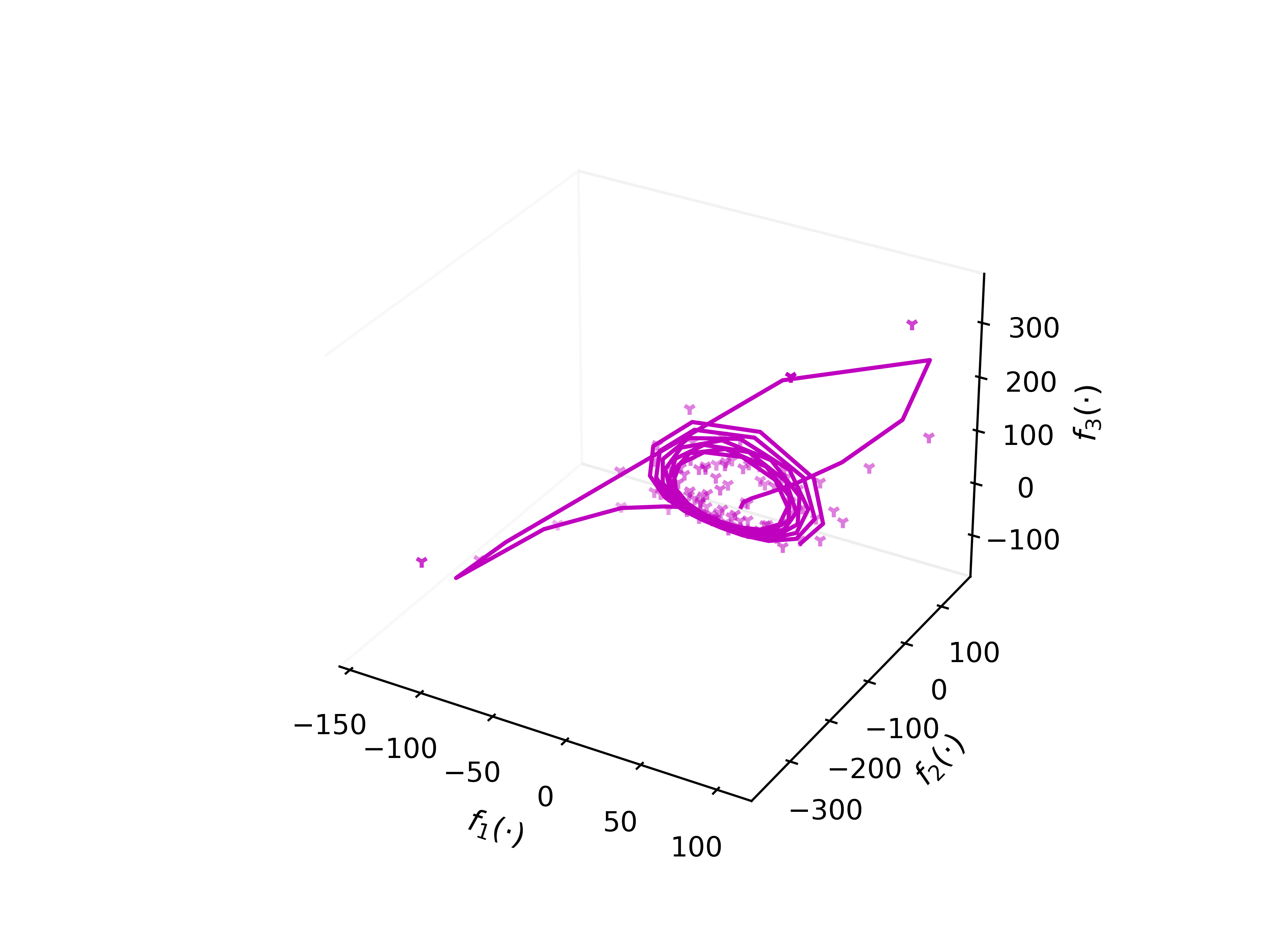}
\caption{fitting to derivative function using GD method: fifth guess}
\end{subfigure}
\caption{Lorenz's system example: marked plots are original data, solid lines are estimated data}\label{fit_Lorenz}
\end{figure}

\begin{table}\centering
\caption{Error analysis for Lorenz's system example using NR method}\label{EA_second_example_NR_method}
\begin{tabular}{|c|c|c|c|c|c|}
\hline
variable&bias&MAPE&MAE&RMSE&R$^2$\\ \hline
$x_1$&0.0392&0.2008&1.1157&1.3487&0.9498\\
$x_2$&0.1503&0.2391&1.3254&1.6295&0.9420\\
$x_3$&0.5704&0.0727&1.4757&1.8549&0.9380\\
$f_1(\cdot)$&1.0479&5.1133&14.3894&17.8915&0.7762\\
$f_2(\cdot)$&0.7459&1.1711&13.4830&17.2733&0.9036\\
$f_3(\cdot)$&-0.3850&2.4390&16.7952&22.1846&0.8856\\
\hline
\end{tabular}
\end{table}

\subsubsection{Comparison to nonlinear LS method}
In this comparison study, in order to apply the constrained-NLS method, the parameters are bounded to be inside this constraint
\begin{equation}
2\le a \le 30
\end{equation}
where all three parameters  are lying inside the bound. 

However, with the same initial guessed parameters, the optimal solution is found as 
\begin{equation}
\bar{a}=\begin{bmatrix}
8.2146&29.8385&2
\end{bmatrix}.\label{eq39}
\end{equation}
The obtained parameters $\bar{a}$ in equation (\ref{eq39}) by using constrained NLS method are farther than the optimal parameters resulted by the proposed methods from the original parameters. The fitting to the data set is plotted in figure \ref{fit_NLS_Lorenz} and fitting coefficients are shown in Table \ref{EA_second_example_NSL_method}.
Contrained NLS method produces bias, MAPE, MAE and RMSE farther from 0, while the proposed methodology and algorithm \ref{algorithm} has closer bias, MAPE, MAE, and RMSE to 0. For a goodness to fit R$^2$, the proposed methodologies and algorithms give closer to 1, even 1 for derivative function 1. It can be drawn that the proposed methodologies and algorithms \ref{algorithm} and \ref{algorithm2} is superior to the constrained NSL method. 

\begin{figure}
\centering
\begin{subfigure}[h!]{0.5\textwidth}
\includegraphics[width=\textwidth]{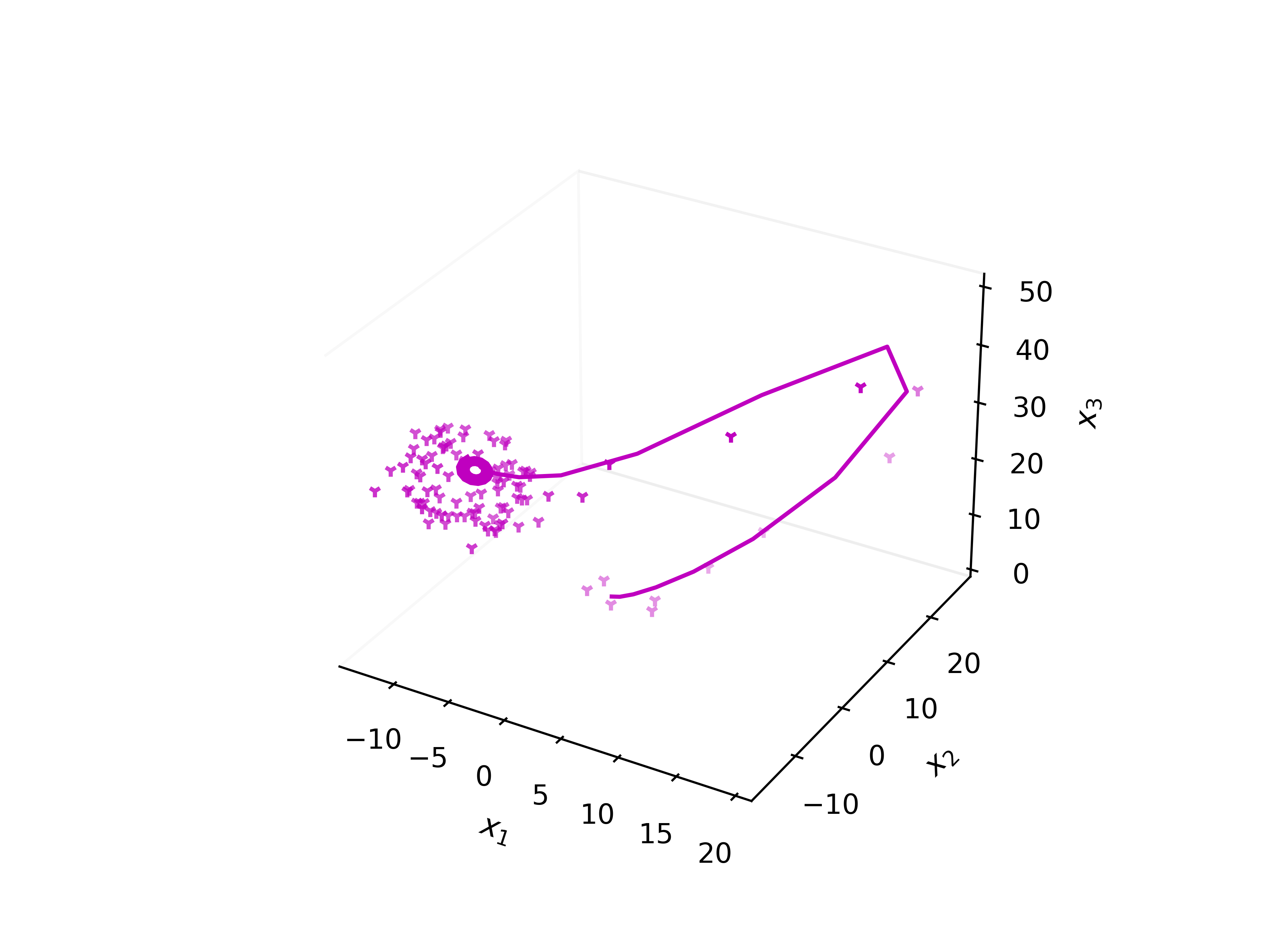}
\caption{fitting to data}
\end{subfigure}~
\begin{subfigure}[h!]{0.5\textwidth}
\includegraphics[width=\textwidth]{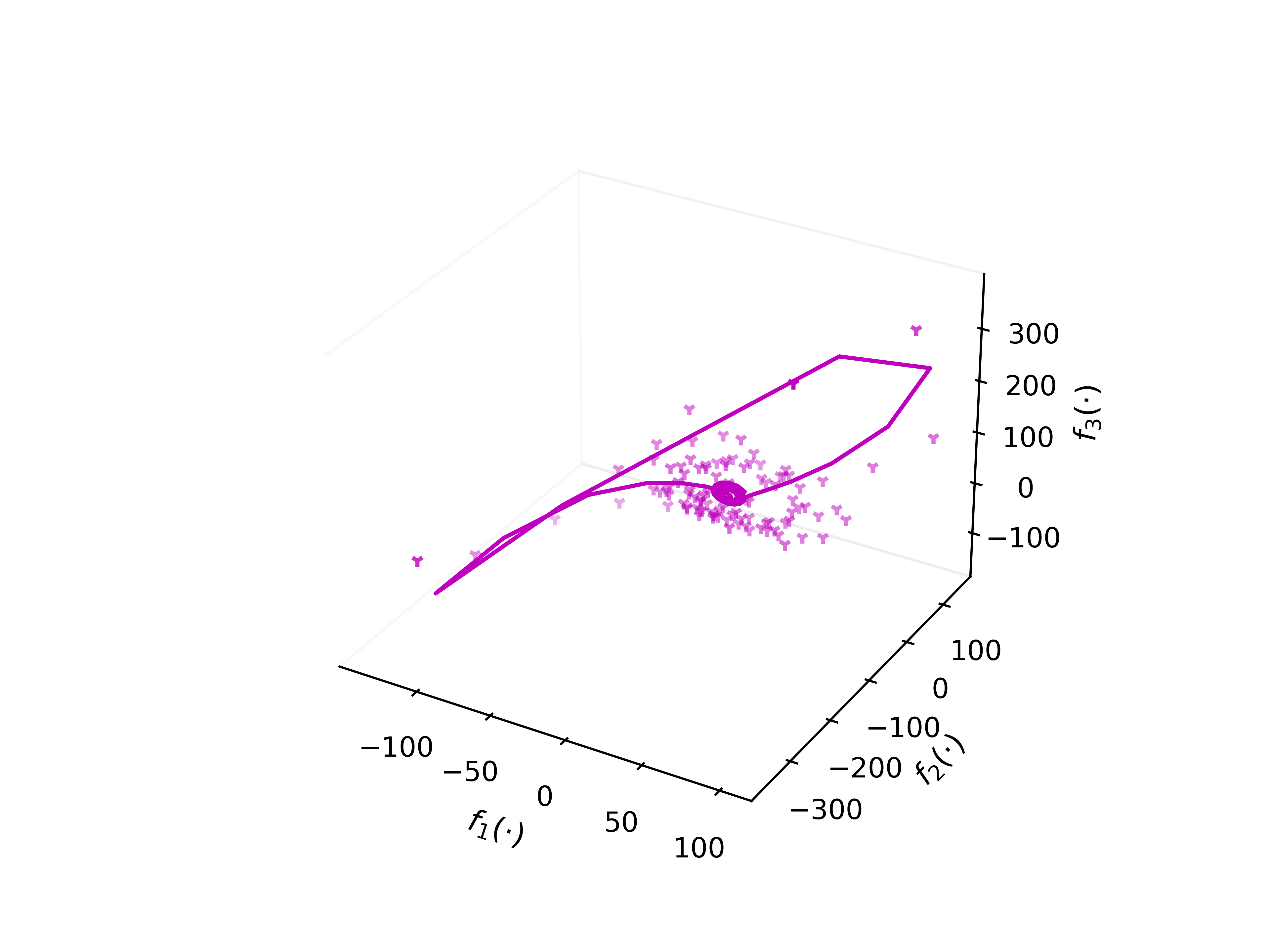}
\caption{fitting to derivative function}
\end{subfigure}
\caption{Lorenz's system example NLS method: marked plots are original data, solid lines are estimated data by using constrained NLS method}\label{fit_NLS_Lorenz}
\end{figure}

\begin{table}\centering
\caption{Error analysis for Lorenz's system example using proposed method}\label{EA_second_example_NSL_method}
\begin{tabular}{|c|c|c|c|c|c|}
\hline
variable&bias&MAPE&MAE&RMSE&R$^2$\\ \hline
$x_1$&-0.5613&0.2924&1.8619&2.2451&0.8609\\
$x_2$&-0.3759&0.3940&2.3256&2.8221&0.8260\\
$x_3$&-2.1314&0.1638&3.5015&4.2435&0.6758\\
$f_1(\cdot)$&1.5081&2.1949&19.6961&25.3391&0.5510\\
$f_2(\cdot)$&1.8151&1.2384&26.3919&32.8529&0.6511\\
$f_3(\cdot)$&-0.9230&1.3472&30.2180&39.2456&0.6421\\
\hline
\end{tabular}
\end{table}

\subsection{Fitting an activator-inhibitor system}
In this example, let an activator-inhibitor system data set $(t_d,x^d_k),\; d=1,\cdots, N\; k=1,2$ be fitted to the following equation (\ref{fe03}) 
\begin{equation}\label{fe03}
\begin{split}
\frac{dx_1}{dt}= &f_1(t,x,a_l)=\frac{1+a_1x^2_1}{1+x^2_1+a_2x_2}-x_1,\\
\frac{dx_2}{dt}= &f_2(t,x,a_l)=a_3(a_4x_1+x_0-x_2),\\
x_0= &0,
\end{split}
\end{equation}
where $a_l, l=1,\cdots,5$ are basal parameters of activator-inhibitor, which are to be estimated based on the available data $(t_d,x^d_k)$ and using the algorithms \ref{algorithm} and \ref{algorithm2}. Clearly, this example contains nonlinear in both parameters and states. \\
\subsubsection{Implementation NR method}\label{NRapply}
For estimating these basal parameters $a$, the vector $E(t,x,a_i)$ is presented as
\begin{equation}
E(t,x,a_i)=\begin{bmatrix}
f_1(t_d,x^d,a_i)-x'_1(t_d)\\
f_2(t_d,x^d,a_i)-x'_2(t_d)\\
\vdots
\end{bmatrix},\quad d=1,\cdots, N-1,
\end{equation}
and updating for each searching iteration. Similarly, the Jacobian matrix is obtained from the derivative of the function with respect to search basal parameters $a_i$ , such as
\begin{equation}\resizebox{0.9\hsize}{!}{$
\nabla_aF(t,x,a_i)=\begin{bmatrix}
\frac{(x^d_1)^2}{1+(x^d_1)^2+a^i_2x^d_2}&-\frac{x^d_2(1+a^i_1x^2_1)}{(1+(x^d_1)^2+a^i_2x^d_2)^2}&0&0\\
0&0&a^i_4x^d_1-x^d_2&a^i_3x^d_1\\
\vdots&\vdots&\vdots&\vdots
\end{bmatrix},\quad d=1,\cdots, N-1$}.
\end{equation}
The Jacobian matrix contains the  searched parameters. Thus, the Jacobian matrix is updated in each searching iteration.    
Here, using the proposed algorithm \ref{algorithm}, the searching basal parameters are updated by applying the same equation (\ref{NR03}). \\
\subsubsection{Implementation GD method}
In the application of the GD method, updated matrices $G(t,x,a_i)=E(t,x,a_i)$ and $\nabla_a G(t,x,a_i)=\nabla_a F(t,x,a_i)$ from this example are taking similar to previous sub sub section \ref{NRapply}. Also, the searching $a_i$ is updated following equation (\ref{GD02}) and algorithm \ref{algorithm2} until condition satisfied.  
\subsubsection{Numerical implementation}
In this activator-inhibitor example of system of ODEs (\ref{fe03}), the data are generated based on basal parameters $
a=[
2\;3\;0.1\;0.4
]^\top
$ numerically with initial condition $x_0=[0.1\;2]^\top$. Also, in this example, the data are added noise. Both methods are initialize with guessed parameters $a_0$  presented in Table \ref{TableIni03}. From these initial parameters,   the derivative vector $E(t,x,a_i)$ and Jacobian matrix $\nabla_aF(t,x,a_i)$ are initially calculated and updated for each iteration $i$. Both algorithms converge very fast which is less than 1 second, except for GD method with second and fifth initial guess. Even having computational expensive, NR method converges to optimal parameters with only less than 11 iterations, however, GD method converges to it's optimal parameter more than 400 iterations. 

\begin{table}\centering
\caption{Initial and convergent parameters and iteration - CPU time of activator-inhibitor example}\label{TableIni03}\footnotesize
\begin{tabular}{|c|c|c|c|c|c|}
\hline
Initial guess&$\begin{bmatrix}
1\\2\\1\\2
\end{bmatrix}$&$\begin{bmatrix}
10\\0\\3\\0.1
\end{bmatrix}$&$\begin{bmatrix}
0\\0\\-10\\0
\end{bmatrix}$&$\begin{bmatrix}
-1\\1\\-10\\9
\end{bmatrix}$&$\begin{bmatrix}
-10\\11\\12\\13
\end{bmatrix}$\\
\hline
\multirow{3}{*}{NR}&6-iter &9-iter&9-iter&7-iter&11-iter\\
&0.0089 s&0.0155 s&0.0098  s&0.0091 s&0.0115 s\\
&\multicolumn{5}{|c|}{$\begin{bmatrix}
2.0117& 3.02825& 0.0948& 0.3737
\end{bmatrix}^\top$}\\
\hline
\multirow{3}{*}{GD}&1210-iter &3897-iter&799-iter&462-iter&3865-iter\\
&0.6407 s&1.9940 s&0.4111 s &2.0562 s&0.2417 s\\
& \multicolumn{5}{|c|}{$\begin{bmatrix}
2.0117& 3.02825& 0.0948& 0.3737
\end{bmatrix}^\top$}\\
\hline
\end{tabular}
\end{table}

By applying the proposed algorithms \ref{algorithm} and \ref{algorithm2}, the optimal parameters are obtained as shown in Table \ref{TableIni03}. 
From these optimal estimated parameters, solution to equation (\ref{fe03}) is obtained numerically. The solution (solid line) and data (tri-right) are plotted in figure \ref{fit_Activator}. It can be noticed that the states are very well fitted to the generated data. As the scattered data points in figure \ref{fit_Activator}.b for derivative function, it is due to the distance between point in respective axis has big different.  
\begin{figure}
\centering
\begin{subfigure}[h!]{0.5\textwidth}
\includegraphics[width=\textwidth]{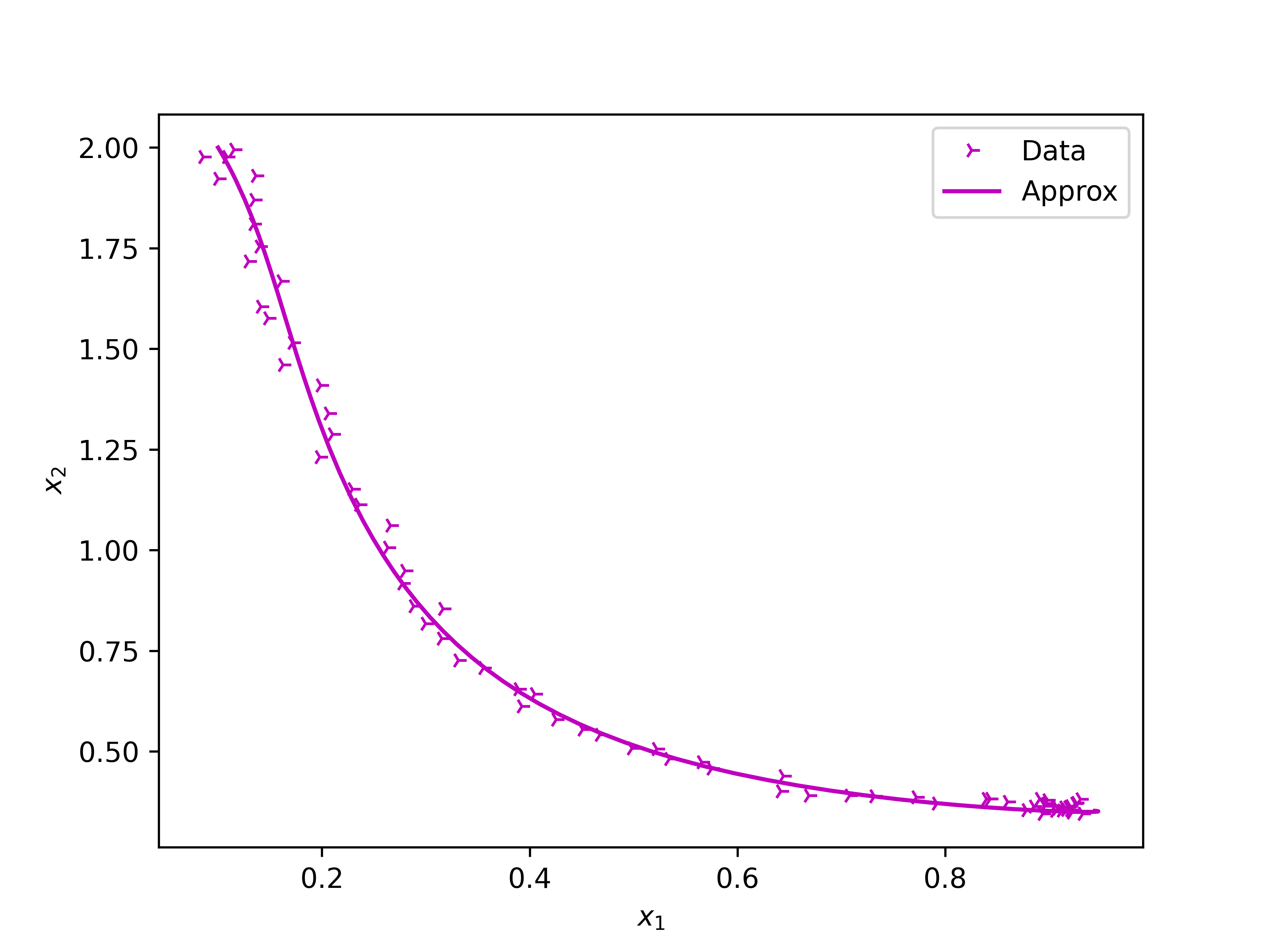}
\caption{fitting to data using NR method}
\end{subfigure}~
\begin{subfigure}[h!]{0.5\textwidth}
\includegraphics[width=\textwidth]{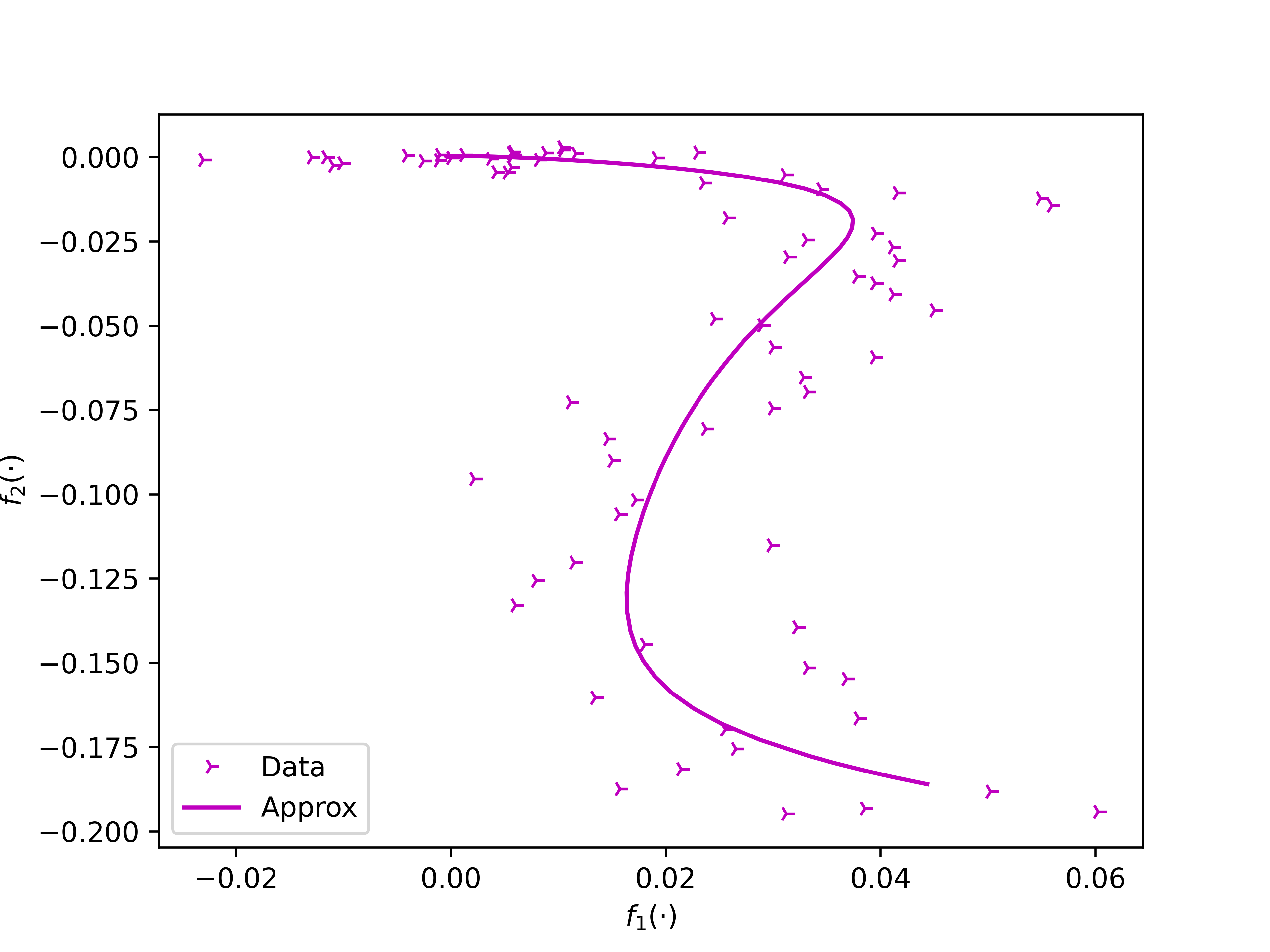}
\caption{fitting to derivative function using NR method}
\end{subfigure}\\
\begin{subfigure}[h!]{0.5\textwidth}
\includegraphics[width=\textwidth]{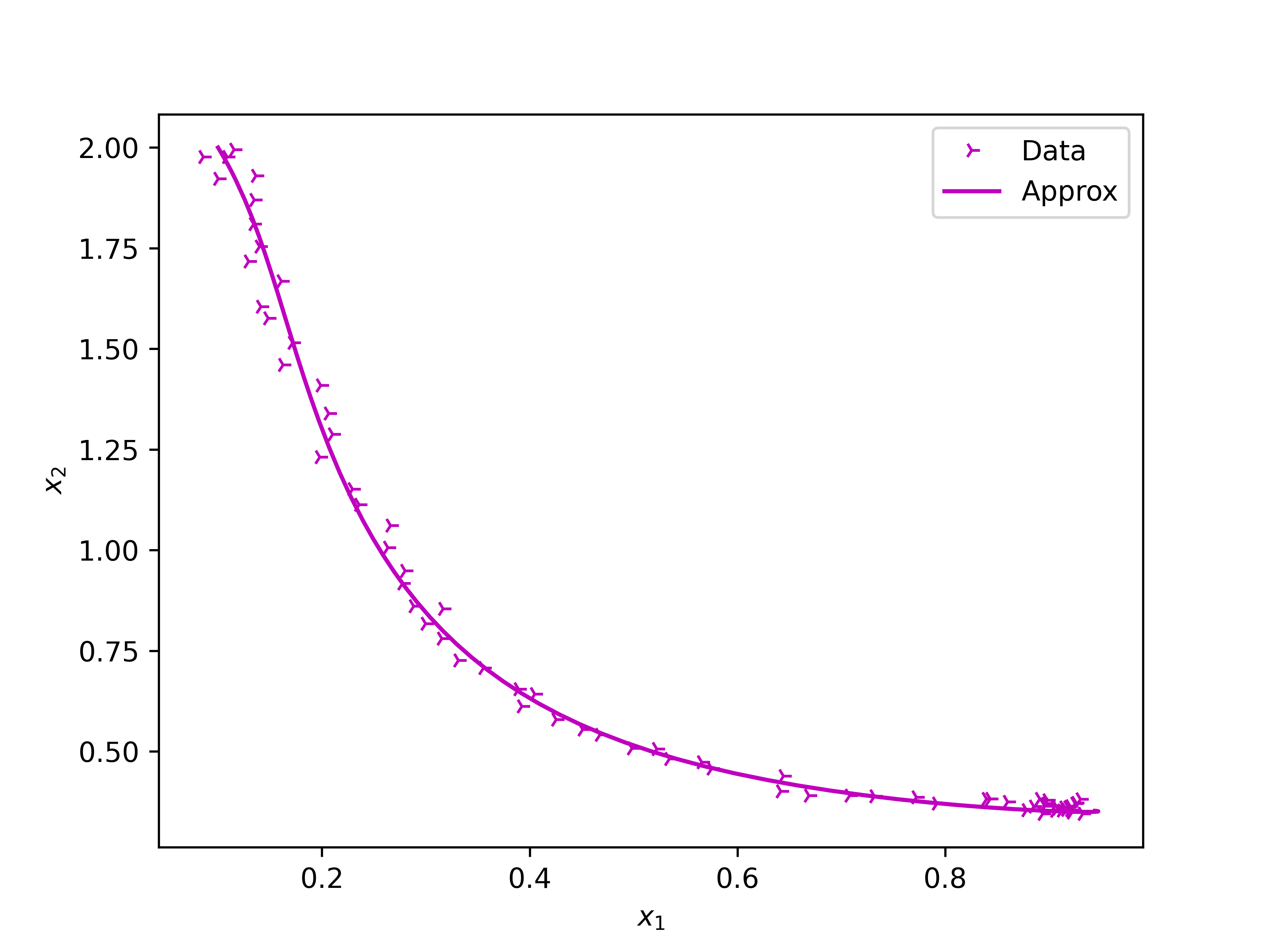}
\caption{fitting to data using GD method}
\end{subfigure}~
\begin{subfigure}[h!]{0.5\textwidth}
\includegraphics[width=\textwidth]{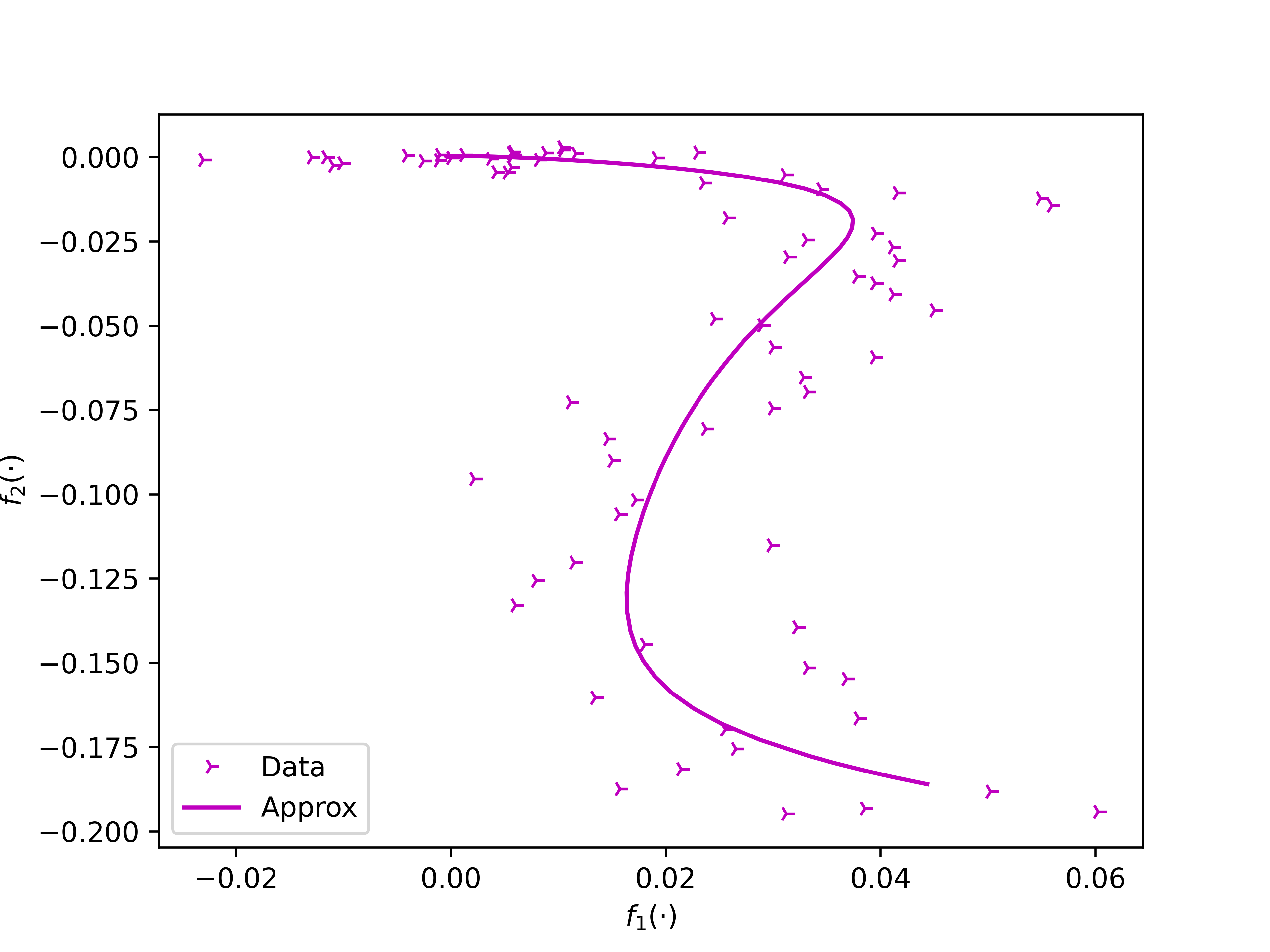}
\caption{fitting to derivative function using GD method}
\end{subfigure}
\caption{Activator-inhibitor example implementation: scatter plots are original data, solid lines are estimated data}\label{fit_Activator}
\end{figure}
The proposed algorithm \ref{algorithm} is 
well fitted of the activator-inhibitor data to the estimated parameters based data generation. It can be seen in Table \ref{EA_Third_example_proposed_method} for error analysis. All analysis show that only for derivative function $f_1(\cdot)$ has MAPE more than 3, and coefficient of determination R$^2$ less than 0.7. However, other analysis precision values gave very well fitting. 

\begin{table}\centering
\caption{Error analysis for activator-inhibitor example using both proposed algorithms}\label{EA_Third_example_proposed_method}
\begin{tabular}{|c|c|c|c|c|c|}
\hline
variable&bias&MAPE&MAE&RMSE&R$^2$\\ \hline
$x_1$&0.0083&0.0603&0.0257&0.0307&0.9903\\
$x_2$&-0.0199&0.0443&0.0285&0.0357&0.9957\\
$f_1(\cdot)$&1.6785e-04&3.0247&0.0089&0.0110&0.6228\\
$f_2(\cdot)$&4.1262e-04&0.5862&0.0038&0.0047&0.9949\\
\hline
\end{tabular}
\end{table}
\subsubsection{Comparison to nonlinear LS method}
The evaluation of the proposed methodology and algorithm \ref{algorithm} is also compared to constrained NLS method for the activator-inhibitor example of activator-inhibitor. For this constrained NLS method, the searching parameters are bounded as 
\begin{equation}
0\le a \le 4.
\end{equation}
It is clear that the solution should be inside the bound as the original parameter is laying in between. Again, the same initial guessed parameter $a_0$ and the same initial condition $x_0$ are used in searching the parameter. The constrained NLS method ends up in a solution for 
\begin{equation}
\bar{a}=\begin{bmatrix}
1.8599&
 2.5681&
 0.0521&
 0.2358
\end{bmatrix}.
\end{equation}
The obtained parameters have error to original parameters for $a_1$ is around 7\%, $a_2$ is around 14.4\%, $a_3$ is around 47.9\%, $a_4$ is around 41.05\%. While, the proposed methodologies and algorithms  obtains the parameters closer to original parameters. This clearly is that the proposed methodologies and algorithms \ref{algorithm} and \ref{algorithm2} are outperfomed of the constrained NLS method. \\
Also, it can be seen on the parameters obtained from the fitting of constrained NLS method to the data in figure \ref{fit_Activator_NLS}. It can be observed that the fitting is far from the data for both states and derivative functions. The coefficients of fitting such bias, MAPE, MAE, RMSE (from Table \ref{EA_Third_example_NLS_method}) are farther from 0, while the proposed method has such coefficients closer to 0 (Table \ref{EA_Third_example_proposed_method}). It can be concluded that the proposed methodologies and algorithms \ref{algorithm} and \ref{algorithm2} is exceeded the constrained NLS method. 

\begin{figure}
\centering
\begin{subfigure}[h!]{0.5\textwidth}
\includegraphics[width=\textwidth]{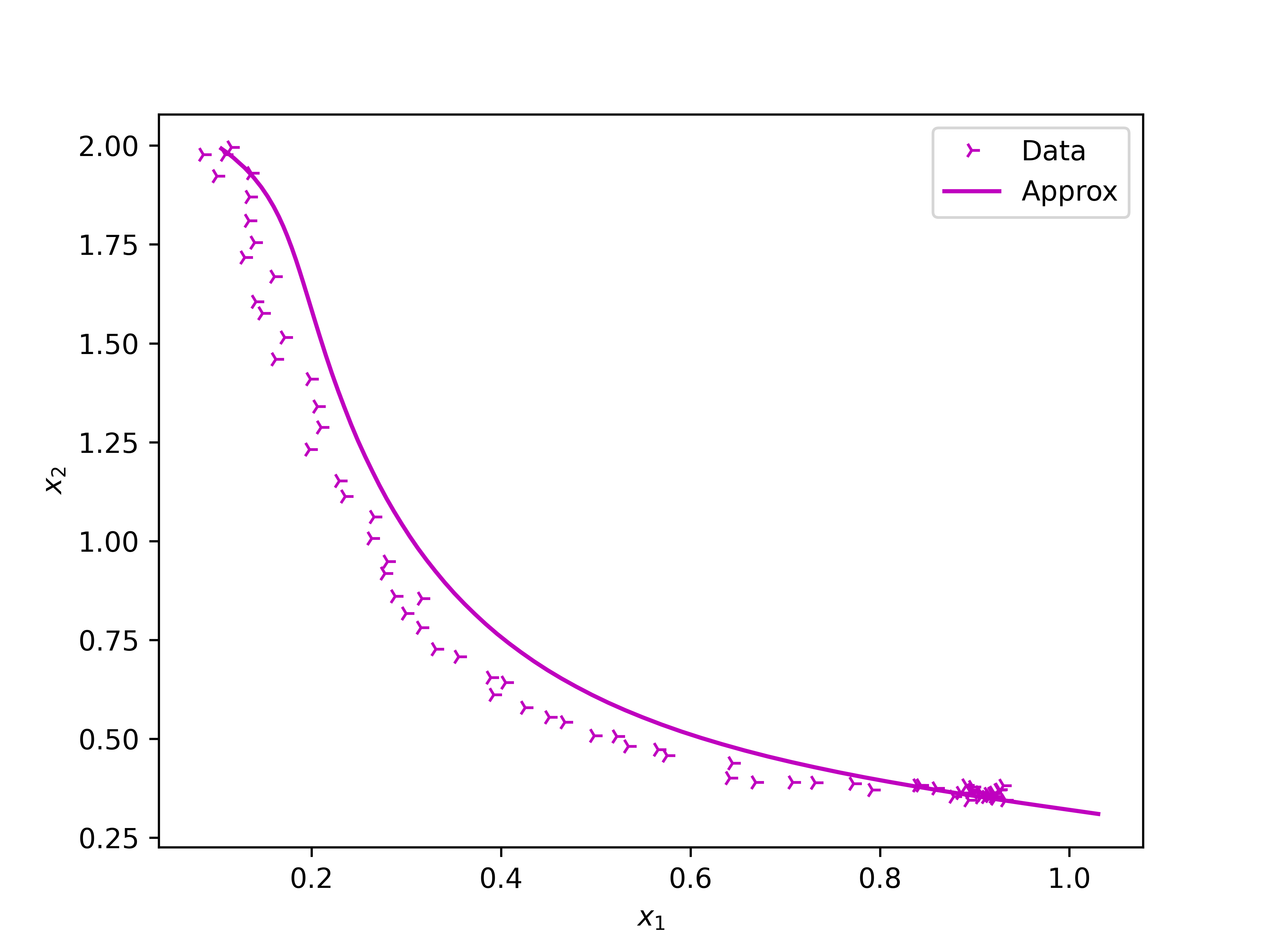}
\caption{fitting to data}
\end{subfigure}~
\begin{subfigure}[h!]{0.5\textwidth}
\includegraphics[width=\textwidth]{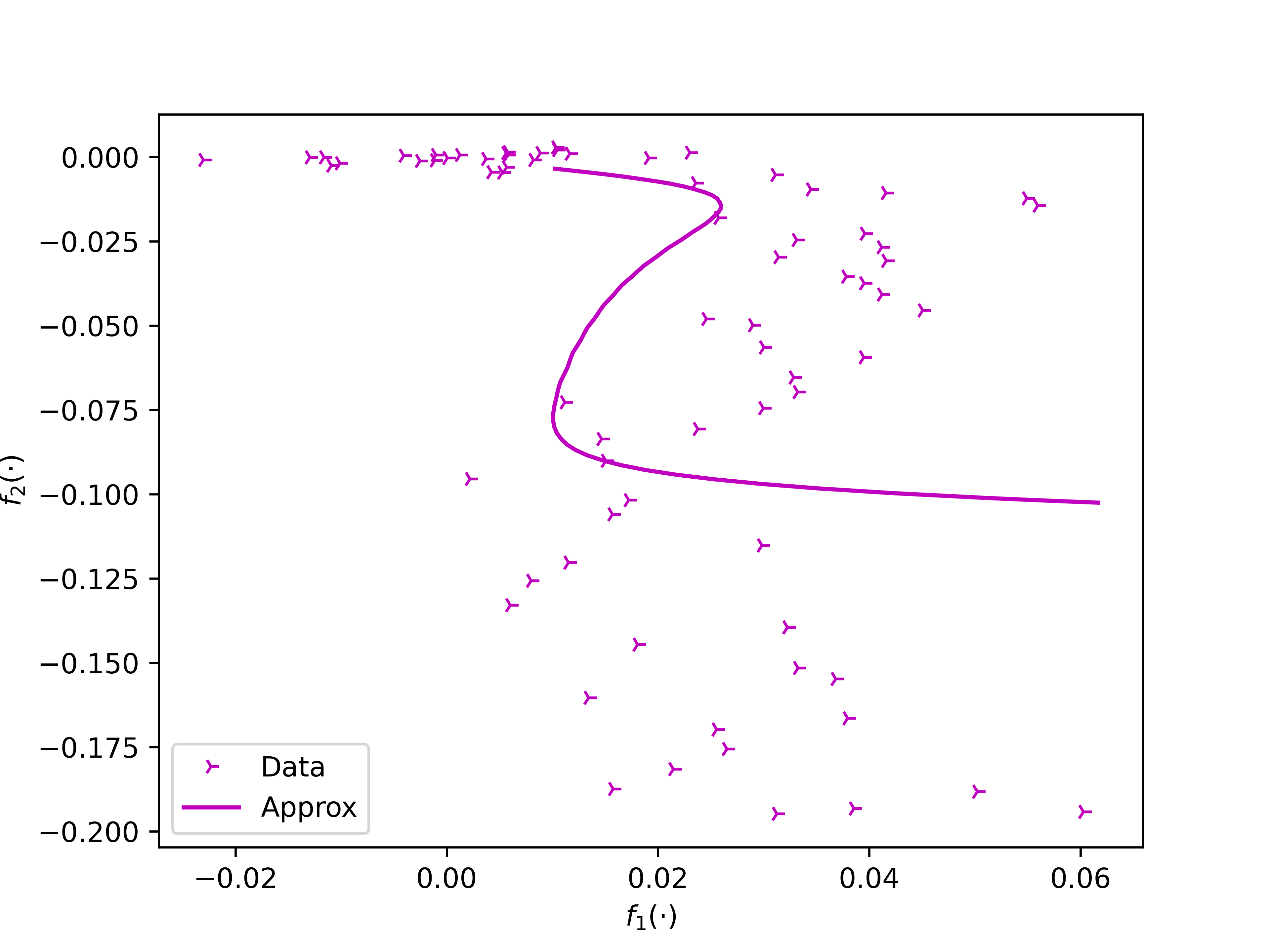}
\caption{fitting to derivative function}
\end{subfigure}
\caption{Activator-inhibitor example: marked plots are original data, solid lines are estimated data by using NLS}\label{fit_Activator_NLS}
\end{figure}

\begin{table}\centering
\caption{Error analysis for activator-inhibitor example using NLS method}\label{EA_Third_example_NLS_method}
\begin{tabular}{|c|c|c|c|c|c|}
\hline
variable&bias&MAPE&MAE&RMSE&R$^2$\\ \hline
$x_1$&0.1150&0.2125&0.1275&0.1723&0.6963\\
$x_2$&-0.2376&0.3971&0.2435&0.2927&0.7120\\
$f_1(\cdot)$&-5.7007e-04&7.0178&0.0173&0.0202&0.2692\\
$f_2(\cdot)$&-0.0101&8.5119&0.0266&0.0369&0.6857\\
\hline
\end{tabular}
\end{table}
\section{Conclusion}\label{conclusion}
In this paper, we proposed novel methodologies for parameter estimation in systems of ODEs using the NR and GD methods. By leveraging discrete derivatives and Taylor expansion, we presented a direct approach to estimate parameters in both linear and nonlinear ODE models. Additionally, we expanded stochastic versions — Stochastic Newton-Raphson (SNR) and Stochastic Gradient Descent (SGD)— to handle large-scale systems more efficiently, reducing computational cost without compromising accuracy.
Through numerical examples, including population dynamics, chaotic Lorenz systems, and activator-inhibitor models, the proposed methods were shown to outperform the traditional NLS method. The SNR method demonstrated rapid convergence and high accuracy, while the SGD method proved robust in managing complex, chaotic systems, though at the cost of higher iteration counts. The stochastic versions further enhanced the scalability of the methods, making them suitable for big data applications.
Overall, the proposed approaches provide flexible and efficient solutions for parameter estimation in ODE systems, offering significant improvements in both accuracy and computational performance compared to NLS. These methods are especially valuable for complex, nonlinear models and large datasets, making them applicable to a wide range of scientific and engineering problems. 

\bibliographystyle{unsrtnat}
\bibliography{Fit_Ref}  

\begin{thebibliography}{36}
\providecommand{\natexlab}[1]{#1}
\providecommand{\url}[1]{\texttt{#1}}
\expandafter\ifx\csname urlstyle\endcsname\relax
  \providecommand{\doi}[1]{doi: #1}\else
  \providecommand{\doi}{doi: \begingroup \urlstyle{rm}\Url}\fi

\bibitem[Hwang and Seinfeld(1972)]{Hwang1972}
M.~Hwang and J.~H. Seinfeld.
\newblock A new algorithm for the estimation of parameters in ordinary
  differential equations.
\newblock \emph{AIChE Journal}, 18\penalty0 (1):\penalty0 90--93, 1972.

\bibitem[Gavalas(1973)]{Gavalas}
George~R. Gavalas.
\newblock A new method of parameter estimation in linear systems.
\newblock \emph{AIChE Journal}, 19\penalty0 (2):\penalty0 214--222, 1973.

\bibitem[Howland and Vaillancourt(1961)]{Howland61}
J.~L. Howland and R.~Vaillancourt.
\newblock A generalized curve-fitting procedure.
\newblock \emph{Journal of the society for industrial and applied mathematics},
  9\penalty0 (2):\penalty0 165--168, 1961.

\bibitem[Strebel(2023)]{Strebel}
Oliver Strebel.
\newblock A preprocessing method for parameter estimation in ordinary
  differential equations.
\newblock \emph{Chaos Soliton and Fractals}, 57:\penalty0 93--104, 2023.

\bibitem[Aydogmus and Tor(2021)]{Aydogmus}
Ozgur Aydogmus and Ali~Hakan Tor.
\newblock A modified multiple shooting algorithm for parameter estimation in
  odes using adjoint sensitivity analysis.
\newblock \emph{Applied Mathematics and Computation}, 390:\penalty0 125644,
  2021.

\bibitem[Xu et~al.(2023)Xu, Wong, and Sang]{Xu}
Mingwei Xu, Samuel~WK Wong, and Peijun Sang.
\newblock A bayesian collocation integral method for parameter estimation in
  ordinary differential equations.
\newblock \emph{arXiv preprint arXiv:2304.02127}, 2023.

\bibitem[Li et~al.(2005)Li, Osborne, and Prvan]{Li2005}
Zhengfeng Li, Michael~R. Osborne, and Tania Prvan.
\newblock Parameter estimation of ordinary differential equations.
\newblock \emph{IMA Journal of Numerical Analysis}, 25\penalty0 (2):\penalty0
  264--285, 2005.

\bibitem[Varah(1982)]{Varah1982}
J.~M. Varah.
\newblock A spline least squares method for numerical parameter estimation in
  differential equations.
\newblock \emph{SIAM Journal on Scientific and Statistical Computing},
  3\penalty0 (1):\penalty0 10.1137/0903003, 1982.

\bibitem[Brunel et~al.(2014)Brunel, Clairon, and D'Alch\'{e}-Buc]{Brunel2014}
Nicolas J.~B. Brunel, Quentin Clairon, and Florence D'Alch\'{e}-Buc.
\newblock Parametric estimation of ordinary differential equations with
  orthogonality conditions.
\newblock \emph{Journal of the American Statical Association}, 109\penalty0
  (505):\penalty0 173 -- 185, 2014.

\bibitem[Cao et~al.(2012)Cao, Huang, and Wu]{Cao2012}
Jiguo Cao, Jianhua~Z. Huang, and Hulin Wu.
\newblock Penalized nonlinear least squares estimation of time-varying
  parameters in ordinary differential equations.
\newblock \emph{Journal of computational and graphical statistics}, 21\penalty0
  (1):\penalty0 42--56, 2012.

\bibitem[Aidala(1977)]{Aidala-1977}
Vincent~J. Aidala.
\newblock Parameter estimation via the kalman filter.
\newblock \emph{IEEE Transactions on Automatic Control}, 22:\penalty0 471--472,
  1977.
\newblock ISSN 0018-9286,1558-2523.
\newblock \doi{10.1109/tac.1977.1101518}.
\newblock URL \url{http://doi.org/10.1109/tac.1977.1101518}.

\bibitem[Mboup(2008)]{Mboup2008}
Mamadou Mboup.
\newblock Parameter estimation for signals described by differential equations.
\newblock \emph{Applicable Analysis}, 88\penalty0 (1):\penalty0 29--52, 2008.

\bibitem[Peifer and Timmer(2007)]{Peifer2007}
M.~Peifer and J.~Timmer.
\newblock Parameter estimation in ordinary differential equations for
  biological processes using the method of multiple shooting.
\newblock \emph{IET Systems Biology}, 1\penalty0 (2):\penalty0 78--88, 2007.

\bibitem[Ramsay et~al.(2007)Ramsay, Hooker, Campbell, and Cao]{Ramsay2007}
J.~O. Ramsay, G~Hooker, D~Campbell, and J~Cao.
\newblock Parameter estimation for differential equations: a generalized
  smoothing approach.
\newblock \emph{Journal of the Royal Statistical Society Series B: Statistical
  Methodology}, 69\penalty0 (5):\penalty0 741--796, 2007.

\bibitem[Baake et~al.(1992)Baake, Baake, Bock, and Briggs]{Baake1992}
E.~Baake, M.~Baake, H.~G. Bock, and K.~M. Briggs.
\newblock Fitting ordinary differential equations to chaotic data.
\newblock \emph{Physical Review A}, 45:\penalty0 5524--5529, Apr 1992.
\newblock \doi{10.1103/PhysRevA.45.5524}.
\newblock URL \url{https://link.aps.org/doi/10.1103/PhysRevA.45.5524}.

\bibitem[Timmer(2000)]{Timmer2000}
J.~Timmer.
\newblock Parameter estimation in nonlinear stochastic differential equations.
\newblock \emph{Chaos, Solitons \& Fractals}, 11\penalty0 (15):\penalty0
  2571--2578, 2000.

\bibitem[Mbalawata et~al.(2013)Mbalawata, S\"{a}rkk\"{a}, and
  Haario]{Mbalawata2013}
Isambi~S. Mbalawata, Simo S\"{a}rkk\"{a}, and Heikki Haario.
\newblock Parameter estimation in stocastic differential equations with markov
  chain monte carlo and nonlinear kalman filter.
\newblock \emph{Computational statistics}, 28:\penalty0 1195 -- 1223, 2013.

\bibitem[Chen et~al.(2021)Chen, Cheng, Arvind~Gupta, and Xu]{Chen2021}
Yu~Chen, Jin Cheng, Huaxiong~Huang Arvind~Gupta, and Shixin Xu.
\newblock Numerical method for parameter inference of systems of nonlinear
  ordinary differential equations with partial observations.
\newblock \emph{Royal Society Open Science}, 8:\penalty0 210171, 2021.

\bibitem[Papavasiliou and Ladroue(2011)]{Papavasiliou2011}
Anastasia Papavasiliou and Christophe Ladroue.
\newblock Parameter estimation for rough differential equations.
\newblock \emph{The Annals of Statistics}, 39\penalty0 (4):\penalty0 2047 --
  2073, 2011.

\bibitem[Singer(2002)]{Singer2002}
Hermann Singer.
\newblock Parameter estimation of nonlinear stochastic differential equations:
  Simulated maximum likelihood versus extended kalman filter and it\^{o}-taylor
  expansion.
\newblock \emph{Journal of computational and graphical statistics}, 11\penalty0
  (4):\penalty0 972 -- 995, 2002.

\bibitem[Nielsen et~al.(2000)Nielsen, Madsen, and Young]{Nielsen2000}
Jan~Nygaard Nielsen, Henrik Madsen, and Peter~C. Young.
\newblock Parameter estimation in stochastic differential equations: an
  overview.
\newblock \emph{Annual Reviews in Control}, 24:\penalty0 83--94, 2000.

\bibitem[Xun et~al.(2013)Xun, Cao, Mallick, Maity, and Carroll]{Xun2013}
Xiaolei Xun, Jiguo Cao, Bani Mallick, Arnab Maity, and Raymond~J. Carroll.
\newblock Parameter estimation of partial differential equations models.
\newblock \emph{Journal of the American statitical association}, 108\penalty0
  (503):\penalty0 1009--1020, 2013.

\bibitem[Mehrkanoon et~al.(2014)Mehrkanoon, Mehrkanoon, and
  Suykens]{Mehrkanoon2014}
Siamak Mehrkanoon, Saeid Mehrkanoon, and Johan A.~K. Suykens.
\newblock Parameter estimation of delay differential equations: an
  integration-free ls-svm approach.
\newblock \emph{Communications in nonlinear science and numerical simulation},
  19\penalty0 (4):\penalty0 830--841, 2014.

\bibitem[Romero et~al.(2021)Romero, Ortega, and Bobtsov]{Romero2021}
Jose~Guadalupe Romero, Romeo Ortega, and Alexey Bobtsov.
\newblock Parameter estimation and adaptive control of euler-lagrange systems
  using the power balance equation parameterisation.
\newblock \emph{International journal of control}, 96\penalty0 (2):\penalty0
  475--487, 2021.

\bibitem[Dua(2011)]{Dua2011}
Vivek Dua.
\newblock An artificial neural network approximation based decomposition
  approach for parameter estimation of system of ordinary differential
  equations.
\newblock \emph{Computer \& Chemical Engineering}, 35\penalty0 (3):\penalty0
  545--553, 2011.

\bibitem[Dua and Dua(2012)]{Dua2012}
Vivek Dua and Pinky Dua.
\newblock A simultaneous approach for parameter estimation of a system
  ofordinary differential equations, using artificial neural
  networkapproximation.
\newblock \emph{Industrial and Engineering Chemistry Research}, 51:\penalty0
  1809--1814, 2012.

\bibitem[Bradley and Boukouvala(2021)]{Bradley2021}
William Bradley and Fani Boukouvala.
\newblock Two-stage approach to parameter estimation of differential equations
  using neural odes.
\newblock \emph{Industrial and Engineering Chemistry Research}, 60:\penalty0
  16330--16344, 2021.

\bibitem[Jamili and Dua(2021)]{Jamili}
Elnaz Jamili and Vivek Dua.
\newblock Parameter estimation of partial differential equations using
  artificial neural network.
\newblock \emph{Computers \& Chemical Engineering}, 147:\penalty0 107221, 2021.

\bibitem[Akman et~al.(2018)Akman, Akman, and Schaefer]{Akman}
Devin Akman, Olcay Akman, and Elsa Schaefer.
\newblock Parameter estimation in ordinary differential equations modeling via
  particle swarm optimization.
\newblock \emph{Journal of applied mathematics}, 2018\penalty0 (Article ID
  9160793), 2018.

\bibitem[Neuberger(1985)]{Neuberger1985}
J.~W. Neuberger.
\newblock Steepest descent and differential equations.
\newblock \emph{Journal of the Mathematical Society of Japan}, 37\penalty0
  (2):\penalty0 187--195, 1985.

\bibitem[Matei et~al.(2022)Matei, Zhenirovskyy, de~Kleer, and Maxwell]{Matei}
Ion Matei, Maksym Zhenirovskyy, Johan de~Kleer, and John Maxwell.
\newblock Improving the efficiency of gradient descent algorithms applied to
  optimization problems with dynamical constraints.
\newblock \emph{arXiv}, 2208.12834v1, 2022.

\bibitem[Tu et~al.(2020)Tu, Rong, and Chen]{Tu2020}
Quan Tu, Yingjao Rong, and Jing Chen.
\newblock Parameter identification of arx models based on modified momentum
  gradient descent algorithm.
\newblock \emph{Complexity}, 2020:\penalty0 Article ID 9537075, 2020.

\bibitem[Barzilai and Bowein(1988)]{Barzilai}
J~Barzilai and J.~M. Bowein.
\newblock Two-point step size gradient methods.
\newblock \emph{IMA Journal of Numerical Analysis}, 8\penalty0 (1):\penalty0
  141--148, 1988.

\bibitem[Dai(2003)]{Dai2003}
Y.~H. Dai.
\newblock Alternate step gradient method.
\newblock \emph{Optimization}, 52\penalty0 (4-5):\penalty0 395--415, 2003.

\bibitem[Raydan and Svaiter(2002)]{Raydan}
M.~Raydan and B.~F Svaiter.
\newblock Relaxed steepest descent and cauchy-barzilai-borwein method.
\newblock \emph{Computational Optimization and Applications}, 21\penalty0
  (2):\penalty0 155--167, 2002.

\bibitem[Hao(2021)]{Hao}
Wenrui Hao.
\newblock A gradient descent method for solving a system of nonlinear
  equations.
\newblock \emph{Applied Mathematics Letters}, 112:\penalty0 106739, 2021.

\end{thebibliography}

\end{document}